\documentclass[aos,preprint]{imsart}

\RequirePackage[OT1]{fontenc}
\RequirePackage{amsthm,amsmath}
\RequirePackage[numbers]{natbib}
\RequirePackage[colorlinks,citecolor=blue,urlcolor=blue]{hyperref}

\usepackage{amssymb,amsfonts,enumerate,color,ifthen,hyperref}
\usepackage{aliascnt,bbm}
\def\rset{\mathbb R}

\newcommand{\pscal}[2]{\left\langle#1,#2\right\rangle}

\newcommand{\eqdef}{\ensuremath{\stackrel{\mathrm{def}}{=}}}

\def\Xset{\mathcal{Z}} % Espace d 'etat
\def\e{\mathcal{E}}

\def\N{\mathcal{N}}
\def\M{\mathcal{M}}

\def\G{\mathcal{G}}

\def\A{\mathcal{A}}

\def\H{\mathcal{H}}

\newcommand{\Sa}{\mathcal{S}}

\def\PP{\mathbb{P}} % proba
\newcommand{\CPP}[3][]
{\ifthenelse{\equal{#1}{}}{{\mathbb P}\left(\left. #2 \, \right| #3 \right)}{{\mathbb P}_{#1}\left(\left. #2 \, \right | #3 \right)}}
\def\PE{\mathbb{E}} % esperance
\newcommand{\CPE}[3][]
{\ifthenelse{\equal{#1}{}}{{\mathbb E}\left[\left. #2 \, \right| #3 \right]}{{\mathbb E}_{#1}\left[\left. #2 \, \right | #3 \right]}}

\def\L{\mathcal{L}} % espace des fonctions

\newcommand{\normfro}[1]{\left\Vert#1\right\Vert_{\textsf{F}}}

\newcommand{\tnorm}[1]{\left\vert\!\left\vert\!\left\vert#1\right\vert\!\right\vert\!\right\vert}
\newcommand{\nnorm}[1]{\left\Vert#1\right\Vert_2}

\def\r{\textsf{r}}

\usepackage{bm}%bold matH

\theoremstyle{plain}
\newtheorem{theorem}{Theorem}
\newtheorem{assumption}{H\hspace{-3pt}}
\newtheorem{assumptionB}{B\hspace{-3pt}}
\newtheorem{assumptionC}{C\hspace{-3pt}}

\newaliascnt{proposition}{theorem}

\aliascntresetthe{proposition}

\newaliascnt{lemma}{theorem}
\newtheorem{lemma}[lemma]{Lemma}
\aliascntresetthe{lemma}
\newaliascnt{corollary}{theorem}

\aliascntresetthe{corollary}

\theoremstyle{definition}
\newaliascnt{definition}{theorem}

\aliascntresetthe{definition}

\newaliascnt{remark}{theorem}
\newtheorem{remark}[remark]{Remark}
\aliascntresetthe{remark}

\newaliascnt{example}{theorem}

\aliascntresetthe{example}

\def\rmd{\mathrm{d}}

\def\1{\mathbbm{1}}

\arxiv{arXiv:0000.0000}

\startlocaldefs
\numberwithin{equation}{section}
\theoremstyle{plain}

\endlocaldefs

\begin{document}

\begin{frontmatter}
\title{On the contraction properties of some high-dimensional quasi-posterior distributions}
\runtitle{high-dimensional quasi-posterior distributions}

\begin{aug}
\author{\fnms{Yves F.} \snm{Atchad\'e}\thanksref{t1}\ead[label=e1]{yvesa@umich.edu}}

%\thankstext{t1}{Some comment}
\thankstext{t1}{Supported in part by NSF grant DMS 1228164 and DMS 1513040}
\runauthor{Yves F. Atchad\'e}

\affiliation{University of Michigan}

\address{Yves F. ATCHAD\'E\\
University of Michigan\\
Department of Statistics\\
1085 South  University\\
Ann Arbor, 48109, MI, United States\\
\printead{e1}
%\phantom{E-mail:\ }\printead*{e2}
}

\end{aug}

\begin{abstract}
We study the contraction properties of a quasi-posterior distribution $\check\Pi_{n,d}$ obtained by combining a quasi-likelihood function and a sparsity inducing prior distribution on $\rset^d$, as both $n$ (the sample size), and $d$ (the dimension of the parameter) increase. We derive some general results that highlight a set of sufficient conditions under which $\check\Pi_{n,d}$ puts increasingly high probability on sparse subsets of $\rset^d$, and contracts towards  the true value of the parameter.  We apply these results to the analysis of logistic regression models, and  binary graphical models, in high-dimensional settings. For the logistic regression model, we shows that for well-behaved design matrices, the posterior distribution contracts at the rate $O(\sqrt{s_\star\log(d)/n})$, where $s_\star$ is the number of non-zero components of the parameter. For the binary graphical model, under some regularity conditions, we show that a quasi-posterior analog of the neighborhood selection of \cite{meinshausen06} contracts in the Frobenius norm at the rate $O(\sqrt{(p+S)\log(p)/n})$, where $p$ is the number of nodes, and $S$ the number of edges of the true graph.
\end{abstract}



\end{frontmatter}

\section{Introduction}\label{sec:intro}
Let $\Xset^{(n)}$ denote a sample space equipped with a reference sigma-finite measure denoted $\rmd z$. The upper script $n$ represents the sample size. Let $Z$ be a $\Xset^{(n)}$-valued random variable that we model as having distribution $\PP^{(n)}_{\theta}$ given a parameter $\theta\in\rset^d$. We assume that $\PP^{(n)}_{\theta}$ has a density $f_{n,\theta}$: $\PP^{(n)}_{\theta}(\rmd z) =f_{n,\theta}(z)\rmd z$. Let $\Pi$ be a prior distribution on $\rset^d$. The resulting posterior distribution for learning the parameter $\theta$ is the random probability measure 
\[A\mapsto \frac{\int_A f_{n,\theta}(Z)\Pi(\rmd \theta)}{\int_{\rset^d} f_{n,\theta}(Z)\Pi(\rmd \theta)},\;\;A\;\mbox{meas.}\subseteq\rset^d.\]
In practice, many inference problems are best tackled using quasi-likelihood (or pseudo-likelihood) functions. In the Bayesian framework, this leads to a quasi-Bayesian inference. Let $(\theta,z)\mapsto q_{n,\theta}(z)$ denote a jointly measurable function such that  $0<\int_{\rset^d} q_{n,\theta}(z)\Pi(\rmd \theta)<\infty$, almost surely $[\rmd z]$. Substituting $q_{n,\theta}$ in place of $f_{n,\theta}$ yields the quasi-posterior (QP) distribution 
\begin{equation}\label{general:quasi:post}
\check\Pi_{n,d}(A\vert Z)\eqdef \frac{\int_A q_{n,\theta}(Z)\Pi(\rmd \theta)}{\int_{\rset^d} q_{n,\theta}(Z)\Pi(\rmd \theta)},\;\;A\subseteq\rset^d.\end{equation}
Although $\check\Pi_{n,d}$ is not a posterior distribution in the usual sense, it possesses the property that it is a probability distribution obtained by tilting a prior distribution using a likelihood-like function. Hence, to the extent that the quasi-likelihood function $\theta\mapsto q_{n,\theta}(Z)$ contains information about the true value of the parameter $\theta$, one can expect the same from the quasi-posterior distribution (\ref{general:quasi:post}), in which case valid inferential procedures can be derived using $\check\Pi_{n,d}$. This idea is perhaps best seen by noting that  (\ref{general:quasi:post}) is a solution of the minimization 
\[ \min_{\mu\ll\Pi} \left[ -\int_{\rset^d} \log q_{n,\theta}(Z)\mu(\rmd\theta) +\textsf{KL}(\mu\vert \Pi)\right],\]
where $\textsf{KL}(\mu\vert \Pi) \eqdef \int_{\rset^d}\log(\rmd\mu/\rmd\Pi)\rmd\mu$ is the KL-divergence between $\mu$ and $\Pi$, and where the minimization is over all probability measures that are absolutely continuous with respect to the prior $\Pi$. We refer to \cite{zhang:06a} for more details (and in particular to Proposition 5.1 of that paper for a proof of the above statement). The implication of this result is that, under appropriate regularity conditions, one can expect the QP distribution to concentrate around the maximizer of the function $\theta\mapsto\log q_{n,\theta}(Z)$, provided that the prior distribution does not prevent it. The goal of this paper is to formalize this idea for a class of statistical models.

As pointed out to us by a referee, QP distributions are commonly used in the PAC-Bayesian framework to aggregate estimators (\cite{mcallester:99,catoni:01,dalalyan:tsybakov:07,alquier:lounici:11,lounici:14}). However in this literature the emphasis is typically on the estimators, not on the QP distributions themselves. An influential work on quasi-Bayesian procedures is  \cite{chernozhukov:hong03}, which subsequently led to the development of quasi-Bayesian inference in semi-parametric modeling, particularly models arising from moment and conditional moment restrictions (\cite{,liao:jiang:11,yang:he:2012,kato:13,li:jiang:14}). Approximate Bayesian computation (ABC) methods (see e.g. \cite{marin:etal:12} and the references therein) are also  popular quasi-Bayesian procedures.

The present paper is motivated by the idea that quasi-Bayesian inference holds a great potential for dealing with high-dimensional statistical models. For some of these models, a likelihood-based inference is  intractable, and this has impeded somewhat the applicability of the Bayesian framework in this area. However, M-estimation procedures that maximizes various quasi/pseudo-likelihood functions are often readily available.  Using the quasi-Bayesian framework, these quasi-likelihood functions can be easily employed to derive tractable quasi-Bayesian procedures. %This is particularly appealing, since in many of these problems, rigorous inference from these M-estimators are technically difficult to produce.
%This type of applications is the main motivation of this work. 

%\subsection{Main contributions}
We study the behavior of the QP distribution (\ref{general:quasi:post}) when the prior distribution $\Pi$ is  given by
\begin{equation}\label{prior}
\Pi(\rmd \theta) = \sum_{\delta\in\Delta_d}\pi_\delta \Pi(\rmd \theta\vert\delta),\end{equation}
for a discrete distribution $\{\pi_\delta,\delta\in\Delta_d\}$ on $\Delta_d\eqdef\{0,1\}^d$, and a sparsity inducing prior $\Pi(\rmd \theta\vert\delta)$ on $\rset^d$, that we build as follows. Given $\delta$, the components of $\theta$ are independent, and for $1\leq j\leq d$,
\begin{equation}\label{basic:prior}
\theta_j\vert \delta\sim \left\{\begin{array}{cc} \textsf{Dirac}(0) & \mbox{ if } \delta_j=0\\ \textsf{Laplace}(\rho) & \mbox{ if } \delta_j=1\end{array}\right.,\end{equation}
where $\textsf{Dirac}(0)$ is the Dirac measure on $\rset$ with full mass at $0$, and $\textsf{Laplace}(\rho)$ denotes the Laplace distribution with parameter $\rho>0$. The marginal prior distribution of $\theta_j$ implied by (\ref{basic:prior}) belongs to the class of  spike-and-slab priors (\cite{mitchell:beauchamp:88}).

We work under the assumption that $Z\sim \PP^{(n)}_{\theta_\star}$ for some $\theta_\star\in\rset^d$. When $d$ is assumed fixed and $n\to\infty$, it is known from the initial work of \cite{chernozhukov:hong03} that $\check\Pi_{n,d}$ concentrates around $\theta_\star$, and is asymptotically Gaussian (when properly scaled). Infinite-dimensional extensions of such results have recently been studied (\cite{liao:jiang:11,florens:simoni:12,kato:13}). The present paper focus on the case where $\check \Pi_{n,d}$ arises from a high-dimensional parametric model with the sparsity inducing prior (\ref{prior}-\ref{basic:prior}), and the results that we derive substantially extend previous works by \cite{castillo:etal:14,li:jiang:14}. More precisely, we derive a general result (Theorem \ref{thm2}) that highlights the key determinants that control the convergence and convergence rate of $\check\Pi_{n,d}$ towards $\theta_\star$. The theorem is  obtained by combining ideas from \cite{castillo:etal:14} together with a general methodology for studying high-dimensional M-estimators synthesized in \cite{negahbanetal10}, as well as an important technical result by \cite{kleijn:vaart:06} on the existence of test functions.

We apply these results to the Bayesian analysis of high-dimensional logistic regression models.  We derive a non-asymptotic result (Theorem \ref{thm1:logit}) that shows that for large $d$, and appropriately large sample size $n$, the resulting posterior distribution $\check\Pi_{n,d}$ puts a high probability on sparse subsets of $\rset^d$, and contracts towards the true value of the parameter $\theta_\star$ as $n,d\to\infty$, at the rate
\[O\left(\sqrt{\frac{s_\star\log(d)}{n}}\right),\]
where $s_\star=\|\theta_\star\|_0$. The constant in the big-O notation depends crucially on some smallest restricted eigenvalues of the Fisher information matrix of the model. 

We also apply the results to a quasi-Bayesian inference of high-dimensional binary graphical  models. Discrete graphical models are known to pose significant difficulties due to the intractable nature of the likelihood function. A very successful frequentist approach to deal with large graphical models is the neighborhood selection method of \cite{meinshausen06} initially proposed for Gaussian graphical models, and extended to the Ising model by \cite{ravikumaretal10}. We analyze a quasi-Bayesian version of neighborhood selection applied to binary graphical models. We show that as $n,p\to\infty$ (where $p$ is the number of nodes in the graph), provided that $n$ is sufficiently large, the QP distribution obtained from neighborhood selection contracts towards the true model parameter $\theta_\star$ in the Frobenius norm at the rate
\[O\left(\sqrt{\frac{(p+S)\log(p)}{n}}\right),\]
where $S$ is the number of edges in the graph defined by $\theta_\star$. This convergence rate is the same as in the Gaussian case with a full likelihood inference (\cite{banerjee:ghosal13}), and compares very well with the best existing frequentist results. For instance \cite{sun:zhang:13} shows that the scaled g-Lasso version of neighborhood selection in the Gaussian case converges at the rate $O\left(s_\star\sqrt{\log(d)/n}\right)$ in the spectral norm, where $s_\star$ is the maximum degree of the graph defined by $\theta_\star$. In general, faster convergence rate can be achieved if one is only interested in components of the matrix.  To illustrate this we analyze the contraction of $\check\Pi_{n,d}$ in the norm $\tnorm{\theta}\eqdef \max_j \|\theta_{\cdot j}\|_2$, where $\theta_{\cdot j}$ is the $j$-th column of $\theta$. We show that in this norm, the QP distribution obtained from neighborhood selection contracts towards $\theta_\star$ at the rate
\[O\left(\sqrt{\frac{s_\star\log(p)}{n}}\right),\] 
where here $s_\star$ is the maximum degree of the graph defined by the true parameter $\theta_\star$. Furthermore, the sample size $n$ required for this result to hold is milder, and comparable to the sample size requirement in simple high-dimensional logistic regressions.

An important issue not addressed in this work is how to obtain Monte Carlo samples from the QP distribution (\ref{general:quasi:post}). It is well known that posterior and quasi-posterior distributions built from discrete-continuous mixture priors as in (\ref{prior})-(\ref{basic:prior}) are computational difficult to handle with standard Markov Chain Monte Carlo algorithms. However there has been some recent progress, including the STMaLa of \cite{schrecketal14}, or the Moreau approximation approach of the author developed in \cite{atchade:15a}. We point the reader to these works for more details and some additional references. Further discussion of computational methods can be in \cite{castillo:etal:14}.

The remainder of the paper is organized as follows. First we close the introduction with some notation that will be used throughout the paper. Section \ref{sec:contract} develops a general analysis of the QP distribution  $\check\Pi_{n,d}$. The applications to logistic regression models and binary graphical models is discussed in Section \ref{sec:logit}. The proof of Theorem \ref{thm2} is presented in Section \ref{sec:proofs}, while the remaining proofs are gathered in the supplementary material \cite{suppl}.

\subsection{Notation}\label{sec:notation}
For an integer $d\geq 1$, we equip the Euclidean space $\rset^d$ with its usual Euclidean inner product $\pscal{\cdot}{\cdot}$, associated norm $\|\cdot\|_2$, and its Borel sigma-algebra. We set $\Delta_d\eqdef\{0,1\}^d$.  We will also use the following norms on $\rset^d$: $\|\theta\|_1\eqdef \sum_{j=1}^d|\theta_j|$, $\|\theta\|_0\eqdef\sum_{j=1}^d \textbf{1}_{\{|\theta_j|>0\}}$, and $\|\theta\|_\infty\eqdef \max_{1\leq j\leq d}|\theta_j|$.

For $\delta\in\Delta_d$, $\mu_{d,\delta}$ denotes the product measure  on $\rset^d$ defined as 
\[\mu_{d,\delta}(\rmd \theta)\eqdef \prod_{j=1}^d\nu_{\delta_{j}}(\rmd \theta_{j}),\]
 where  $\nu_{0}(\rmd z)$ is the Dirac mass at $0$, and $\nu_1(\rmd z)$ is the Lebesgue measure on $\rset$. For $\theta,\theta'\in\rset^d$, $\theta\cdot\theta'\in\rset^d$ denotes the component-wise product of $\theta$ and $\theta'$:  $(\theta\cdot\theta')_j=\theta_j\theta_j'$, $1\leq j\leq d$. And for $\delta\in\Delta_d$, we set $\delta^c\eqdef1-\delta$, that is $\delta_j^c\eqdef1-\delta_j$, $1\leq j\leq d$. For $\theta\in\rset^d$, the sparsity structure of $\theta$ is the element $\delta\in\Delta_d$ defined as $\delta_j=\textbf{1}_{\{|\theta_j|>0\}}$, $1\leq j\leq d$.

Throughout the paper $e$ denotes the Euler number, and ${m \choose q}$ is the combinatorial number $m!/(q!(m-q)!)$. For $x\in\rset$, the notation $\lceil x\rceil$ represents the smallest integer larger of equal to $x$, and $\textsf{sign}(x)$ is the sign of $x$ ($\textsf{sign}(x)=1$ if $x>0$, $\textsf{sign}(x)=-1$ if $x<0$, and $\textsf{sign}(x)=0$ if $x=0$). Finally, for $\theta\in\rset^d$, and $A\subset\rset^d$, $\theta+A\eqdef\{\theta+u,\;u\in A\}$.

\section{Contraction properties of the quasi-posterior distribution $\check\Pi_{n,d}$}\label{sec:contract}
We consider the QP distribution (\ref{general:quasi:post}) on $\rset^d$, with the prior distribution (\ref{prior}-\ref{basic:prior}). Using the notation of Section \ref{sec:notation}, $\check \Pi_{n,d}$ can be written as 
\begin{equation}\label{quasi:post}
\check\Pi_{n,d}(\rmd \theta\vert  Z)\propto q_{n,\theta}(Z) \sum_{\delta\in\Delta_d}  \pi_\delta\left(\frac{\rho}{2}\right)^{\|\delta\|_0}e^{-\rho\|\theta\|_1} \mu_{d,\delta}(\rmd\theta).
\end{equation}

We are interesting in the contraction behavior of $\check\Pi_{n,d}$ for large $n,d$.  We take the usual frequentist view of Bayesian procedures by assuming the following.

\begin{assumption}\label{H0}
There exists $\theta_\star\in\rset^d$ such that $Z\sim \PP^{(n)}_{\theta_\star}(\rmd z)= f_{n,\theta_\star}(z)\rmd z$.
\end{assumption}

We write $\PE^{(n)}$ for the expectation operator with respect to $\PP^{(n)}_{\theta_\star}(\rmd z)$.  We also make the basic assumption that the quasi-likelihood function is log-concave and smooth, and we use the notation $\nabla \log q_{n,u}(z)$ to denote the derivative of the map $\theta\mapsto \log q_{n,\theta}(z)$ at $u$. The $j$-th component of $\nabla \log q_{n,u}(z)$ is written as $(\nabla \log q_{n,u}(z))_j$. 

\begin{assumption}\label{H1}
For all $z\in\Xset^{(n)}$, the map $\theta\mapsto \log q_{n,\theta}(z)$ is concave and differentiable.
% and there exists $\bar L$ such  that for all $z\in\Xset^{(n)}$, $u,v\in\rset^d$,
%\begin{equation}\label{lip}
%\left\|\nabla\log q_{n,u}(z) -\nabla\log q_{n,v}(z)\right\|\leq \bar L\|u-v\|.\end{equation}
\end{assumption}

\medskip
\begin{remark}
The assumption that the function $\theta\mapsto \log q_{n,\theta}(z)$ is concave is imposed mostly for simplicity, and is not crucial to derive the main result (Theorem \ref{thm2}). In fact, this assumption is not used in Theorem \ref{thm2}-(2). However, in the application of Theorem \ref{thm2}, concavity is typically crucial to control the events $\e_n$  that appear in the theorem. 
\end{remark}

%We will also impose the following mild assumption on the slab density $u\mapsto e^{-h(u)}/c_0$.
%\begin{assumption}\label{H2}
%$h(0)=0$, and there exists $0<\rho<\infty$ such that
%\[\left|h(u)-h(v)\right|\leq \rho|u-v|,\;\;u,v\in\rset.\]
%\end{assumption}
%\begin{remark}
%In the examples we work mainly with the Laplace prior $\textsf{Laplace}(\rho)$ with density $(\rho/2)e^{-\rho|z|}$, which clearly satisfies H\ref{H2}. However, most commonly used shrinkage priors satisfies H\ref{H2}. Although H\ref{H2} may not directly suggest it, it will be important to have a control over the Lipschitz constant $\rho$ in H\ref{H2}, and to choose it carefully, as function of $d$ and $n$, in order to guarantee the contraction of the quasi-posterior distribution.
%\end{remark}

Following \cite{castillo:etal:14}, we specify the prior $\{\pi_\delta,\;\delta\in\Delta_d\}$ as follows.
\begin{assumption}\label{H2}
For all $\delta\in\Delta_d$, $\pi_\delta = g_{\|\delta\|_0} {d\choose \|\delta\|_0}^{-1}$, for a discrete distribution $\{g_s,\;0\leq s\leq d\}$, for which there exist positive universal constants $c_1,c_2$, $c_3\geq c_4$  such that
\begin{equation}\label{eq:H2}
\frac{c_1}{d^{c_3}}g_{s-1} \leq g_s\leq \frac{c_2}{d^{c_4}}g_{s-1},\;s = 1,\ldots,d.\end{equation}
\end{assumption}
\begin{remark}\label{rem:pi:delta}
This assumption guarantees that the prior distribution concentrates on sparse subsets of $\rset^d$. Note that $\{g_s\}$ is the distribution of the number of non-zero components produced by the prior. The assumption in (\ref{eq:H2}) guarantees that for $d$ large enough so that $\frac{c_2}{d^{c_4}}<1$, we have $g_s\leq (\frac{c_2}{d^{c_4}})^s g_0$, and the rate $\frac{c_2}{d^{c_4}}$ gets smaller with $d$.

\cite{castillo:etal:12} has several examples of prior distributions that satisfy H\ref{H2}. For instance if, for some hyper-parameter $u>1$, $\textsf{q}\sim \textbf{Beta}(1,d^u)$, and given $\textsf{q}$, we draw independently $\delta_j\sim\textsf{Ber}(q)$, then the marginal distribution of $\delta$ in this case satisfies H\ref{H2}, with $c_1=1/2$, $c_2=1$, $c_3 = u$ and $c_4=u-1$.
\end{remark}

We study the contraction properties of $\check\Pi_{n,d}$ towards $\theta_\star$. 
%The results that we derive are non-asymptotic and can be useful to understand the behavior of a large class of quasi-posterior distributions as both the sample size $n$ and the parameter dimension $d$ grow. It seems natural to view $\check\Pi_{n,d}$  as a special type of mis-specified posterior distribution. This is the approach taken here, and we borrow ideas from the analysis of general Bayesian nonparametric missspecified models as developed by \cite{kleijn:vaart:06}. Although the results obtained  by \cite{kleijn:vaart:06} cannot be applied in our setting, these authors solve a general technical problem that plays a key role in our analysis: they prove the existence of test functions to test a given probability measure against a set of finite measure alternatives.
We  borrow a strategy developed mostly for the analysis of high-dimensional M-estimators, that consists in identifying a ``good" subset $\e_n$ of the sample space $\Xset^{(n)}$ on which the map $\theta\mapsto q_{n,\theta}(Z)$ has good curvature properties (see e.g. \cite{negahbanetal10} for an excellent presentation of these ideas). Using this idea, the task at hand then boils down to controlling the probability of the set $\e_n$ and showing that $\check\Pi_{n,d}$ has good contraction properties when $Z\in\e_n$. To that end, and to shorten notation, we introduce the (Bregman divergence) function 
\[\L_{n,\theta}(z) \eqdef  \log q_{n,\theta}(z)-\log q_{n,\theta_\star}(z)-\pscal{\nabla \log q_{n,\theta_\star}(z)}{\theta-\theta_\star},\;\;\theta\in\rset^d,\;z\in\Xset^{(n)}.\]
This function plays a key role in informing on the curvature of the objective function $\theta\mapsto\log q_{n,\theta}(Z)$ around $\theta_\star$. However, in high-dimensional settings, it is typically not realistic to assume that $\theta\mapsto\log q_{n,\theta}(Z)$  has good curvature on the entire parameter space $\rset^d$. As well explained in \cite{negahbanetal10}, one should look at restrictions of  $\L_{n,\theta}(z)$ to interesting subsets of $\rset^d$.
  
We will use a rate function to express the curvature of $\theta\mapsto\log q_{n,\theta}(Z)$. Throughout the paper, a continuous function $\r:\;[0,\infty)\to [0,\infty)$ is a rate function if $\r$ is strictly increasing, $\r(0)=0$, and $\lim_{x\downarrow 0}\r(x)/x=0$. Given a rate function $\r$, and $a\geq 0$, we define 
\begin{equation}\label{phir}
\phi_\r(a)\eqdef\inf\{x>0:\; \r(z)\geq a z,\; \mbox{for all } z\geq x\},\end{equation}
 with the convention that $\inf\emptyset =+\infty$. The main example of a rate function is  $\r(x) = \tau x^2$, for some $\tau>0$ (for linear regression problems). However, the examples below are related to logistic regression and the rate function $\r(x)=\tau x^2/(1+bx)$ is used.
 
% Given $s\geq 0$, we set
%\[
%\A_0(s)\eqdef \left\{\theta\in\rset^d:\; \|\theta\|_0\leq s\right\},\]
%the set of $s$-sparse elements of $\rset^d$. 
A non-empty subset $\Theta$ of $\rset^d$ is a cone if for all $\lambda\geq 0$, and all $x\in\Theta$, $\lambda x\in\Theta$.  We will say that a cone $\Theta$ is a split cone if $u\cdot x\in\Theta$ for all $x\in\Theta$, and all $u\in\{-1,1\}^d$ (we recall that the notation $u\cdot x$ denotes the component-by-component product). Split cones serve as generalizations of sparse  subsets of $\rset^d$. The archetype example of a split cone is the set of $s$-sparse elements: $\{\theta\in\rset^d:\;\|\theta\|_0\le s\}$. However in some problems, one might have to work with sparse elements with some additional structure, and this motivates the introduction of the split cones. A particularly important example of a split cone is the set of elements of $\rset^d$ with the same sparsity structure as $\theta_\star$:
\begin{equation}\label{Theta_star}
\Theta_\star \eqdef \left\{\theta\in\rset^d:\; \theta_j=0 \mbox{ for all  $j$ s.t. } \theta_{\star j}=0\right\}.
\end{equation}

Another important example of split cone that we will use is the set 
\begin{equation*}
\N\eqdef\left\{\theta\in\rset^d:\; \theta\neq 0,\;\mbox{ and }\; \sum_{j:\;\delta_{\star j}=0}|\theta_j|\leq 7\|\theta\cdot\delta_\star\|_1\right\},\end{equation*}
where $\delta_\star$ denote the sparsity structure of $\theta_\star$: $\delta_{\star j}=\textbf{1}_{\{|\theta_{\star j}|\neq 0\}}$, $1\leq j\leq d$. 

\medskip
Given a rate function $\r$, and a split cone $\Theta\subseteq\rset^d$, we set
\begin{equation}\label{def:e21}
\check \e_{n,1}(\Theta,\r)\eqdef \left\{z\in \Xset^{(n)}:\; \mbox{ for all } \theta\in\theta_\star + \Theta,\; \L_{n,\theta}(z) \leq -\frac{1}{2} \r(\nnorm{\theta-\theta_\star})\right\}.\end{equation}

Here as in classical Bayesian asymptotics, in order to control the normalizing constant of the quasi-posterior distribution, we need a lower bound on the function $\theta\mapsto \L_{n,\theta}(z)$. Again, a restricted version will suffice. For $L\geq 0$, we set
\begin{equation}\label{def:e22}
\hat \e_{n,1}(\Theta,L)\eqdef \left\{z\in \Xset^{(n)}:\; \mbox{ for all } \theta\in\theta_\star + \Theta,\; \L_{n,\theta}(z) \geq -\frac{L}{2} \|\theta-\theta_\star\|_2^2\right\}.\end{equation}
Finally, for $\lambda>0$ we set
\begin{equation}\label{def:e0}
\e_{n,0}(\Theta,\lambda)\eqdef \left\{z\in\Xset^{(n)}:\; \sup_{u\in\Theta,\;\nnorm{u}=1} \left|\pscal{\nabla\log q_{n,\theta_\star}(z)}{u}\right| \leq \frac{\lambda}{2}\right\}.\end{equation}
The main idea behind these definitions is that  on the event $\{Z\in \hat \e_{n,1}(\Theta,L)\cap \check \e_{n,1}(\Theta,\r)\}$ the quasi-log-likelihod function  $\theta\mapsto \log q_{n,\theta}(Z)$ has very nice curvature properties when restricted to the set $\theta_\star + \Theta$. 
The definition of $\e_{n,0}(\Theta,\lambda)$ implies that on the event $\{Z\in\e_{n,0}(\Theta,\lambda)\}$, $\theta_\star$ is close to the maximizer of the map $\theta\mapsto \log q_{n,\theta}(Z)$. Hence the set $\e_{n,0}(\Theta,\lambda)\cap \hat \e_{n,1}(\Theta,L)\cap \check \e_{n,1}(\Theta,\r)$ is our example of a ``good set", and on that set, we expect $\check\Pi_{n,d}(\cdot\vert Z)$ to have good concentration properties around $\theta_\star$. This is the substance of the next result. Before stating the main theorem, we introduce few more notation. For $M>0$, let $\textsf{B}_d(\Theta,M)\eqdef\{\theta\in\theta_\star + \Theta,\;\mbox{ s.t. }\; \nnorm{\theta-\theta_\star}\leq M\}$. For $\epsilon>0$, let $\textsf{D}(\epsilon,\textsf{B}_d(\Theta,M))$ denote the $\epsilon$-packing number of the ball $\textsf{B}_d(\Theta,M)$, defined as the maximal number of points in $\textsf{B}_d(\Theta,M)$ such that the $\nnorm{\cdot}$-distance between any pair of such points is at least $\epsilon$. 

\begin{theorem}\label{thm2}
Assume H\ref{H0}-H\ref{H2}, and set $s_\star\eqdef\|\theta_\star\|_0$. Suppose that $d$ is such that $d^{c_4}\geq 8c_2$. Let  $\bar\Theta\supseteq\Theta_\star$ be a split cone, $\bar L\geq 0,\bar\lambda\geq 0$, and a rate function $\r$ be such that $\bar\epsilon\eqdef \phi_\r\left(2\bar\lambda\right)$ is finite.
\begin{enumerate}
\item Set $\e_n\eqdef \e_{n,0}(\rset^d,\rho)\cap\hat\e_{n,1}(\Theta_\star,\bar L)\cap\check\e_{n,1}(\N,\r)$. Then for any integer $k\geq 0$,
\begin{multline}\label{bound:thm1}
\PE^{(n)}\left[\check\Pi_{n,d}\left(\left\{\theta\in\rset^d:\;\|\theta\|_0\geq s_\star +k\right\}\vert Z\right)\right]
\leq \PP^{(n)}\left[Z\notin \e_{n}\right] \\
+  2e^{\textsf{a}}\left(4+\frac{4\bar L}{\rho^2}\right)^{s_\star} {d\choose s_\star}\left(\frac{4c_2}{d^{c_4}}\right)^k, \end{multline}
where $\textsf{a}=-\frac{1}{2} \inf_{x> 0} \left[\r(x) -4\rho\sqrt{s_\star} x\right]$, if $\N\neq \emptyset$, and $\textsf{a}=0$ if $\N=\emptyset$.
\item  Set $\e_n \eqdef \e_{n,0}(\bar\Theta,\bar\lambda)\cap \hat\e_{n,1}(\Theta_\star,\bar L)\cap\check\e_{n,1}(\bar \Theta,\r)$. For any  $M_0> 2$, 
\begin{multline}\label{thm2:maineq}
\PE^{(n)}\left[\check\Pi_{n,d}\left(\left\{\theta\in\theta_\star+\bar\Theta:\;\nnorm{\theta-\theta_\star} >M_0\bar\epsilon\right\}\vert Z\right)\right]\leq  \PP^{(n)}\left[Z\notin \e_{n}\right] \\
+ \sum_{j\geq 1}\textsf{D}_je^{-\frac{1}{8}\r(\frac{j M_0 \bar\epsilon}{2})}
+2  {d\choose s_\star}\left(\frac{d^{c_3}}{c_1}\right)^{s_\star}\left(1+\frac{\rho^2}{\bar L}\right)^{s_\star}  \sum_{j\geq 1}e^{-\frac{1}{8}\r(\frac{j M_0 \bar\epsilon}{2})} e^{3\rho c_0 j M_0\bar\epsilon},
\end{multline}
where  $\textsf{D}_j \eqdef \textsf{D}\left(\frac{j M_0\bar \epsilon}{2},\textsf{B}_d(\bar\Theta,(j+1)M_0\bar\epsilon)\right)$, and where \\
$c_0 \eqdef \sup_{u\in\bar\Theta}\sup_{v\in\bar \Theta,\;\nnorm{v} = 1} |\pscal{\textsf{sign}(u)}{v}|$.
\end{enumerate}
\end{theorem}
\begin{proof}
See Section \ref{sec:proof:thm2}.
\end{proof}

%\begin{remark}
%In many cases, it is possible to bound the $\epsilon$-packing numbers $\textsf{D}_j$ using for instance the arguments in Example 7.1 of \cite{ghosal:etal:00}, and show that there exists $\mathcal{N}_d$ such that
%and the combinatorial inequality ${n\choose k}\leq e^{k\log(n e/k)}$ (where $e$ is the exponential number), to get
%\[\sup_{j\geq 1}\textsf{D}_j \leq \exp\left(\mathcal{N}_d\right).\]
%Therefore the term $\sum_{j\geq 1}\textsf{D}_je^{-\frac{1}{8}\r(\frac{j M_0 \bar\epsilon}{2})}$ in (\ref{thm2:maineq}) satisfies
%\begin{equation}\label{rmk1:thm2}
%\sum_{j\geq 1}\textsf{D}_je^{-\frac{1}{8}\r(\frac{j M_0 \bar\epsilon}{2})}\leq e^{\mathcal{N}_d} \sum_{j\geq 1}e^{-\frac{1}{8}\r(\frac{j M_0 \bar\epsilon}{2})}.\end{equation}
%\end{remark}

Theorem \ref{thm2}-Part(1) shows that for $\rho$, $\bar L$ and $\r$ such that the event $\{Z\in \e_{n,0}(\rset^d,\rho)\cap\hat\e_{n,1}(\Theta_\star,\bar L)\cap\check\e_{n,1}(\N,\r)\}$ has high probability, one can use the second term on the right-hand side of (\ref{bound:thm1}) to establish that the concentration of the prior on sparse subsets (as assumed in H\ref{H2}) is inherited by the quasi-posterior distribution.  In the logistic regression example below, we show that the term $e^{\textsf{a}}\left(4+\frac{4\bar L}{\rho^2}\right)^{s_\star}$  is  $O(e^{cs_\star\log(d)})$, for some constant $c$. And since ${d \choose s_\star}\leq e^{s_\star\log(de)}$, it follows that for such models the right-side of (\ref{bound:thm1}) becomes small for $k$ of the order of $(c/c_4) s_\star$. The same is true for  linear regression models (\cite{castillo:etal:14}).

Part (2) of the theorem shows that if $\bar\lambda,\bar L$, the split cone $\bar\Theta$ and the rate function $\r$ are well chosen such that the event $\{Z\in \e_{n,0}(\bar\Theta,\bar\lambda)\cap \hat\e_{n,1}(\Theta_\star,\bar L)\cap\check\e_{n,1}(\bar \Theta,\r)\}$ has high probability, then the convergence rate of the quasi-posterior distribution is controlled mainly by the series $\sum_j e^{-\frac{1}{8}\r(\frac{j M_0 \bar\epsilon}{2})}$, and its dependence on $n,d$. In the examples below, we show how the terms on the right-hand of (\ref{thm2:maineq}) can be handled.
%In many cases, it is possible to derive bounds on the $\epsilon$-packing numbers $\textsf{D}_j$, following for instance the arguments in Example 7.1 of \cite{ghosal:etal:00}. Such bounds then allows us to work out simple conditions under which the term $\sum_{j\geq 1}\textsf{D}_je^{-\frac{1}{8}\r(\frac{j M_0 \bar\epsilon}{2})}$ converges to zero with $d$. A similar argument applies to the last term of (\ref{thm2:maineq}), the control of which boils down to comparing $\r(\frac{j M_0 \bar\epsilon}{2})$ and $3\rho c_0 j M_0\bar\epsilon$, as $j\to\infty$.

We note that Part (2) of the theorem controls only the probability of the event $\left\{\theta\in\theta_\star+\bar\Theta:\;\nnorm{\theta-\theta_\star} >M_0\bar\epsilon\right\}$ whereas in most applications we typically want the probability of $\left\{\theta\in\rset^d:\;\nnorm{\theta-\theta_\star} >M_0\bar\epsilon\right\}$. As we will show in the examples below, one can use Part (1) of the theorem to upper bound separately the probability of the event $\left\{\theta\notin \theta_\star+\bar\Theta\right\}$.

%The choice of the Euclidean norm is not crucial in Theorem \ref{thm2}-Part(2), in the sense that one can easily replace the $\ell^2$-norm in the definition of the set $\check\e_n$ in (\ref{def:e21}) by some alternative norm, and obtain the corresponding bound (\ref{thm2:maineq}) in terms of that norm. However such generalization may not be very useful, since in most cases one ultimately rely on Taylor expansion (a technique strongly tied to the Euclidean metric) to deal with the event $\check\e_n(\Theta,\r)$.

Finally, we point out that the upper bounds in (\ref{bound:thm1}) and (\ref{thm2:maineq}) depends in general on $\theta_\star$, typically through $\bar L$ and the rate function $\r$. These terms essentially model the curvature of $\theta\mapsto \log q_{n,\theta}(Z)$ around $\theta_\star$. Our setting thus differs from the linear regression setting where the curvature of $\theta\mapsto \log q_{n,\theta}(Z)$ is constant, and the resulting posterior concentration bounds are uniform in $\theta_\star$ (\cite{castillo:etal:14} Theorem 1 and 2).

%\begin{remark}
%Theorem \ref{thm2bis} can be used to study the contraction properties of $\check\Pi_{n,d}$ in the norm $\nnorm{\cdot}$. However, most existing restricted strong convexity analysis have been done using the Euclidean norm. Extensions of such results in terms of the norm $\nnorm{\cdot}$ will be needed to apply Theorem \ref{thm2bis}, particularly in dealing with the term $\PP\left[Z\notin \e_{n}\right]$. We will not pursue this issue here.
%\end{remark}

\section{Sparse Bayesian logistic regression}\label{sec:logit}
As a first application we study the contraction behavior of a posterior distribution obtained from a high-dimensional logistic regression model, for large values of the sample size $n$ and the dimension $d$. Suppose that $Z_1,\ldots,Z_n$ are independent $0$-$1$ binary random variables and we consider the model
\[\PP(Z_i=1) = \frac{e^{\pscal{x_i}{\theta}}}{1+e^{\pscal{x_i}{\theta}}},\]
for a parameter $\theta\in\rset^d$, where $x_i\in\rset^d$ is a known vector of covariates.
Writing $z=(z_1,\ldots,z_n)$, the likelihood function is then
\[q_{n,\theta}(z) = \exp\left(\sum_{i=1}^n z_i\pscal{x_i}{\theta}-g\left(\pscal{x_i}{\theta}\right)\right),\]
where 
\[g(x) \eqdef  \log(1+e^x),\;\;x\in\rset.\]
Using the  prior distribution given in (\ref{prior})-(\ref{basic:prior}), we consider the posterior distribution
\begin{equation}\label{post:log}
\check\Pi_{n,d}(\rmd\theta\vert Z)\propto \exp\left(\sum_{i=1}^n Z_i\pscal{x_i}{\theta}-g\left(\pscal{x_i}{\theta}\right)\right) \sum_{\delta\in\Delta} \pi_\delta \left(\frac{\rho}{2}\right)^{\|\delta\|_1}e^{-\rho\|\theta\|_1} \mu_{d,\delta}(\rmd\theta).\end{equation}
We make the following assumption that implies H\ref{H0}.
\begin{assumptionB}\label{B1}
$Z_1,\ldots,Z_n$ are independent $0$-$1$ binary random variables, and there exist $\theta_\star\in \rset^d$, $x_1,\ldots,x_n\in\rset^d$, such that
\[\PP(Z_i=1) = \frac{e^{\pscal{x_i}{\theta_\star}}}{1+e^{\pscal{x_i}{\theta_\star}}},\;i=1,\ldots,n.\] \end{assumptionB}

Let $X\in\rset^{n\times d}$ denote the design matrix, where the $i$-th row of $X$ is given by the transpose of $x_i$. We shall write $g'$, and $g^{(2)}$ to denote the first and second derivatives of $g$. Let $W\in\rset^{n\times n}$ be the diagonal matrix with $i$-th diagonal entry given by
\[W_i = g^{(2)}\left(\pscal{x_i}{\theta_\star}\right),\;\;i=1,\ldots,n.\]
We define
\[\underline{\kappa}_1 \eqdef \inf \left\{\frac{\theta'(X'WX)\theta}{n\|\theta\|_2^2}:\;\theta\in\rset^d\setminus\{0\},\;\|\theta\cdot\delta_\star^c\|_1\leq 7\|\theta\cdot\delta_\star\|_1\right\}.\]
For $s\in\{1,\ldots,d\}$,  we define
\begin{multline*}
\bar\kappa_1(s) \eqdef\sup\left\{\frac{\theta'(X'X)\theta}{n\|\theta\|_2^2}:\;\;1\leq \|\theta\|_0\leq s\right\},\;\;\\
\mbox{ and }\;\; \underline{\kappa}_1(s) \eqdef \inf \left\{\frac{\theta'(X'WX)\theta}{n\|\theta\|_2^2}:\;\;1\leq \|\theta\|_0\leq s\right\}.\end{multline*}

We choose the regularization parameter $\rho$ in the prior distribution  (\ref{basic:prior}) as
\begin{equation}\label{rho:logit}
\rho \eqdef  4\|X\|_\infty\sqrt{n\log(d)},
\end{equation}
where $\|X\|_\infty\eqdef \max_{i,j}|X_{ij}|$. We note that $\bar\kappa_1(1) \leq \|X\|_\infty^2$, and $\underline{\kappa}_1(s)\leq \bar\kappa_1(1)/4$, for all $s\geq 1$.

\begin{theorem}\label{thm1:logit}
Assume B\ref{B1} and H\ref{H2}. Choose $\rho$ as in (\ref{rho:logit}). Set $s_\star\eqdef \|\theta_\star\|_0$,
\begin{equation}\label{eq:zeta:logit}
\zeta \eqdef s_\star + \frac{2}{c_4} +\frac{2}{c_4}\left(1 + \frac{64\|X\|_\infty^2}{\underline{\kappa}_1} +\frac{\bar\kappa(s_\star)}{64\|X\|_\infty^2(\log(d))^2}+\frac{\log(4e)}{\log(d)}\right)s_\star,\end{equation}
and $\bar s\eqdef \lceil s_\star + \zeta\rceil$.
%\begin{equation*}\label{rate:logit}
%r_{n,d} \eqdef \frac{16\|X\|_\infty}{\underline{\kappa}_1(\bar s)}\sqrt{\frac{\bar s\log(d)}{n}}.\end{equation*}
If $\underline{\kappa}\eqdef\min(\underline{\kappa}_1,\underline{\kappa}_1(\bar s))>0$, then there exists a universal constant $A<\infty$ such that for all $d$ large enough, and
\begin{equation}\label{cond:nd:logistic:part1}n\geq A\|X\|_\infty^4 \left(\frac{s_\star}{\underline{\kappa}}\right)^2\log(d),\end{equation}
the following statements hold.
\begin{enumerate}
\item
\[\PE^{(n)}\left[\check\Pi_{n,d}\left(\left\{\theta\in\rset^d:\;\|\theta\|_0\geq \zeta\right\}\vert Z\right)\right] \leq \frac{4}{d}.\]
\item There exists a finite constant $M_0>2$ (that depends only on the constants in H\ref{H2}), such that
\[\PE^{(n)}\left[\check\Pi_{n,d}\left(\left\{\theta\in\rset^d:\;\|\theta-\theta_\star\|_2>\frac{M_0\|X\|_\infty}{\underline{\kappa}_1(\bar s)}\sqrt{\frac{\bar s\log(d)}{n}}\right\}\vert Z\right)\right] \leq \frac{12}{d}.\]
\end{enumerate}
\end{theorem}
\begin{proof}
See Section \ref{sec:proof:thm:logit}.
\end{proof}

\medskip
%\begin{remark}
If the dimension $d$ is large, then 
\[\zeta\approx s_\star + \frac{2}{c_4} + \frac{2}{c_4}\left(1 + \frac{64\|X\|_\infty^2}{\underline{\kappa}_1}\right)s_\star.\] 
Therefore, for design matrices $X$ for which the restricted eigenvalues $\underline{\kappa}_1$ and $\underline{\kappa}_1(\bar s)$ of the matrix $n^{-1}X'WX$ are not too small, Theorem \ref{thm1:logit} implies that  most of the probability mass of the posterior distribution is on sparse subsets of $\rset^d$, and the rate of convergence of the posterior distribution (\ref{post:log}) is $O\left(\sqrt{\frac{s_\star\log(d)}{n}}\right)$. The frequentist $\ell^1$-penalized M-estimator for logistic regression has been analyzed by \cite{negahbanetal10} (assuming a random design matrix $X$), and \cite{li:etal:14} (assuming a deterministic design matrix $X$), and is known to converge at the same rate, and under assumptions that are similar to those imposed above. Technically, our approach is closer to \cite{li:etal:14}. The approach of \cite{negahbanetal10} leads to slightly better conditions on the sample size $n$ (they require $n$ to increase linearly in $s_\star$, not quadratically, as in (\ref{cond:nd:logistic:part1})), at the expense of more structure on the design matrix ($X$ is assumed to have i.i.d. rows from a sub-Gaussian distribution and positive definite covariance).
%\end{remark}

\begin{remark}
As pointed out by a referee, one can use the convergence rate in Theorem \ref{thm1:logit}-Part(2) with an argument used in \cite{castillo:etal:12}~Theorem 2.2 to derive a bound on the convergence rate in the $\ell^q$-norm for $q\in(0,2]$:
\[\PE^{(n)}\left[\check\Pi_{n,d}\left(\left\{\theta\in\rset^d:\;\|\theta-\theta_\star\|_q>\frac{M_0\|X\|_\infty (\bar s)^{\frac{1}{q}}}{\underline{\kappa}_1(\bar s)}\sqrt{\frac{\log(d)}{n}}\right\}\vert Z\right)\right] \leq \frac{16}{d}.\]
This follows from the fact that for any $r>0$,
\begin{multline*}
\left\{\theta\in\rset^d:\;\|\theta-\theta_\star\|_q>r\right\}\subseteq \\
\left\{\theta\in\rset^d:\;\|\theta-\theta_\star\|_q>r,\|\theta\|_0\leq \zeta\right\} \bigcup \left\{\theta\in\rset^d:\;\|\theta\|_0>\zeta\right\},\end{multline*}
and by Holder's inequality, for $\theta\in\rset^d$ such that $\|\theta\|_0\leq \zeta$, $\|\theta-\theta_\star\|_0\leq \bar s$, and
\[\|\theta-\theta_\star\|_q \leq \|\theta-\theta_\star\|_2(\bar s)^{\frac{1}{q}-\frac{1}{2}}.\]
Obviously, the same argument can be used with respect to the general bound in Theorem \ref{thm2}, but the resulting bound would be more complicated.
\end{remark}

\begin{remark}
It is interesting to observe that the contraction result given in Theorem \ref{thm1:logit}~Part(2) holds, not in spite of the large dimension $d$, but because $d$ is large. In other words, the result should be viewed as a form of concentration of measure phenomenon for $\check\Pi_{n,d}$ as $d\to\infty$. In particular, Theorem \ref{thm1:logit} should not  be applied to a fixed-dimension case in an attempt to recover standard Bayesian contraction results (fixed $d$, $n\to\infty$). Indeed, note that for $d$ fixed, the prior distribution $\Pi$ in (\ref{prior}-\ref{basic:prior}) with $\rho$ as in (\ref{rho:logit}) converges weakly to a point-mass at $0$ as $n\to\infty$, which is not a good behavior of a prior in fixed-dimensional settings. However with more appropriate prior assumptions, the argument in the proof of Theorem \ref{thm2} can be easily modified to derive convergence rate results that would be applicable to the fixed-dimensional setting. We refer to \cite{ghosh:rama:03} (and the references therein) for a good presentation of finite-dimensional Bayesian asymptotics.
\end{remark}
%\subsection{Numerical illustration}\label{sec:logit:sim}

\section{Quasi-Bayesian inference of large binary graphical models}\label{sec:ising}
As another example, we consider the Bayesian analysis of high-dimensional binary graphical models (sometimes called Ising models). Let $\M_p$ be the space of real-valued $p\times p$ symmetric matrices.  For $\theta\in\M_p$, let $f_\theta$ be the probability mass function  defined on $\{0,1\}^p$ by
\begin{equation}\label{mod:ising}
 f_\theta(x_1,\ldots,x_p) = \frac{1}{Z_\theta}\exp\left(\sum_{j=1}^p\theta_{jj}x_j +\sum_{i<j}\theta_{ij}x_ix_j\right),\;\;x_j\in\{0,1\},\;1\leq j\leq p,\end{equation}
where $Z_\theta$ is the normalizing constant. We consider the problem of estimating $\theta$ under a sparsity assumption, from a matrix $Z\in\rset^{n\times p}$ where each row of $Z$ is an independent realization from $f_{\theta_\star}$ for some sparse $\theta_\star\in\M_p$. This problem has  generated some literature in recent years (\cite{barnejeeetal08,hoefling09, ravikumaretal10,atchade:ejs14,barber:drton:2015} and the references therein), all in the frequentist framework.

The Bayesian estimation of $\theta$ is significantly more challenging because  the normalizing constant $Z_\theta$ are typically intractable, and this leads to posterior distributions that are doubly intractable. %We note that there has been some recent progress on MCMC methods for doubly-intractable posterior distributions (see e.g. \cite{lyneetal14} and the references therein). However, whether these methods can successfully handle problems with high-dimensional parameters remains to be seen. 
In the frequentist literature cited above, the preferred approach for estimating $\theta$ is via penalized pseudo-likelihood maximization, which nicely side-steps the intractable normalizing constants issue. The quasi-Bayesian framework developed in this work can be used to combine these pseudo-likelihood functions with a prior distribution to produce quasi-Bayesian posterior distributions.

The most commonly used pseudo-likelihood function is obtained by taking the product of all the conditional densities in (\ref{mod:ising}). This is an idea that goes back at least to \cite{besag74}. 
%The resulting quasi-likelihood function is 
%\[\bar q_{n,\theta}(Z) = \prod_{j=1}^p \prod_{i=1}^n \frac{\exp\left(Z_{ij}\left(\theta_{jj}+\sum_{k\neq j} \theta_{kj}Z_{ik}\right)\right)}{1+\exp\left(\theta_{jj}+\sum_{k\neq j} \theta_{kj}Z_{ik}\right)},\;\;\theta\in\M_p.\]
Combined with a prior distribution $\Pi$ on $\M_p$, this approach readily yields a quasi-posterior distribution on $\M_p$ that falls in the framework presented above.  Note however that when $p$ is large, say $p\geq 500$, the space $\M_p$ has dimension bigger than $10^5$, and MCMC sampling from this quasi-posterior distribution becomes a daunting and time consuming task. One interesting idea is to break the symmetry and to consider the quasi-likelihood 
\begin{equation}\label{q_ising}
q_{n,\theta}(Z) = \prod_{j=1}^p \prod_{i=1}^n \frac{\exp\left(Z_{ij}\left(\theta_{jj}+\sum_{k\neq j} \theta_{kj}Z_{ik}\right)\right)}{1+\exp\left(\theta_{jj}+\sum_{k\neq j} \theta_{kj}Z_{ik}\right)},\;\;\theta\in\rset^{p\times p}.\end{equation}
Notice that the only difference between $\bar q_{n,\theta}$ and $q_{n,\theta}$ is that the symmetry constraint in $\theta$ is relaxed, that is the parameter space of the map $\theta\mapsto q_{n,\theta}(Z)$ is $\rset^{p\times p}$, not $\M_p$.  However this difference has a huge impact since now $q_{n,\theta}(Z)$ factorizes along the columns of $\theta$. As a result, maximizing a penalized version of  (\ref{q_ising}) is equivalent to solving $p$ independent logistic regression (assuming a separable penalty), and this can be done efficiently in a parallel computing environment. This pseudo-likelihood approach was popularized by the influential paper  \cite{meinshausen06} in the Gaussian case, and extended to the Ising model by \cite{ravikumaretal10}. In a recent work (\cite{atchade:15:c}), the author extended this idea to the Bayesian analysis of large Gaussian graphical models, and analyzed the contraction of the resulting quasi-posterior distribution using Theorem \ref{thm2}. Here we extend the method to the Ising model.

Throughout this section, if $\theta\in\rset^{p\times p}$, $\theta_{\cdot j}\in\rset^p$ denotes the $j$-th column of $\theta$. In view of the discussion above, and for a discrete probability distribution $\{\pi_\delta,\;\delta\in\Delta_p\}$ on $\Delta_p$, and $\rho>0$, we consider the quasi-posterior $\check\Pi_{n,d}$ on $\rset^{p\times p}$ given by
\begin{eqnarray}\label{quasipost:ising}
\check\Pi_{n,d}(\rmd\theta\vert Z) & \propto & q_{n,\theta}(Z)\prod_{j=1}^p \sum_{\delta\in\Delta_p}\pi_\delta \left(\frac{\rho}{2}\right)^{\|\delta\|_0}e^{-\rho\|\theta_{\cdot j}\|_1} \mu_{p,\delta}(\rmd\theta_{\cdot j})\\
& =& \prod_{j=1}^p \check\Pi_{n,d,j}(\rmd\theta_{\cdot j}\vert Z)\;\;\;.\nonumber
\end{eqnarray}
where $\check\Pi_{n,d,j}(\cdot\vert Z)$ is the probability measure on $\rset^p$ given by
\begin{multline*}
\check\Pi_{n,d,j}(\rmd u\vert Z)\propto \prod_{i=1}^n \frac{\exp\left(Z_{ij}\left(u_j+\sum_{k\neq j} u_kZ_{ik}\right)\right)}{1+\exp\left(u_j+\sum_{k\neq j} u_kZ_{ik}\right)}\\
\times \sum_{\delta\in\Delta_p}\pi_\delta \left(\frac{\rho}{2}\right)^{\|\delta\|_0}e^{-\rho\|u\|_1} \mu_{p,\delta}(\rmd u).\end{multline*}

\begin{remark}
One of the limitation of the approach is that the distribution $\check\Pi_{n,d}$ does not necessarily produce symmetric matrices. However, because of the contraction properties discussed below, typical realizations of $\check\Pi_{n,d}$ will be close to be symmetric. Furthermore, from a practical viewpoint, one can easily remedy a broken symmetry using various symmetrization rules as suggested for instance in \cite{meinshausen06}.
\end{remark}

We make the following assumptions.

\begin{assumptionC}\label{C1}
The rows of $Z\in\rset^{n\times p}$ are independent $\{0,1\}^p$-valued random variables with common probability mass function $f_{\theta_\star}$, for some $\theta_\star\in \M_p$.
\end{assumptionC}
 We define 
\[s_{\star j}\eqdef \|\theta_{\star \cdot j}\|_0,\;\;\mbox{ and }\;\; s_\star \eqdef \max_{1\leq j\leq p}s_{\star j}.\]
Hence $s_\star$ is the maximum degree of the undirected graph encoded by $\theta_\star$. The sparsity structure of $\theta_\star$ is the matrix $\delta_{\star}\in\{0,1\}^{p\times p}$ defined as $\delta_{\star, jk} =\textbf{1}_{\{|\theta_{\star jk}|>0\}}$. For $X\sim f_{\theta_\star}$, and $1\leq j\leq p$, we define 
\[X_{(j)}\eqdef (X_1,\ldots,X_{j-1},1,X_{j+1},\ldots,X_p)\in\rset^p,\] 
(viewed as  a column vector), and 
\[\mathcal{H}^{(j)} \eqdef \PE\left[g^{(2)}\left(\pscal{\theta_{\star \cdot j}}{X_{(j)}}\right)X_{(j)} X_{(j)}'\right].\]
We set
\begin{multline}\label{def:kappa:ising}
\underline{\kappa}_2(s) \eqdef \inf_{1\leq j\leq p}\inf\left\{ \frac{u'\H^{(j)} u}{\|u\|_2^2},\;u\in\rset^p\setminus\{0\},\;\;\|u\|_0\leq s\right\},\;\;\mbox{ and }\;\\
 \;\; \underline{\kappa}_2 \eqdef \inf_{1\leq j\leq p}\inf\left\{ \frac{u'\H^{(j)}u}{\|u\|_2^2},\;\;u\in\rset^p\setminus\{0\},\;\;\sum_{k:\;\delta_{\star kj}\neq 0}|u_k| \leq 7\sum_{k:\;\delta_{\star kj}= 0}|u_k|\right\}.
 \end{multline}

\begin{remark}
It is easy to verify that 
\[\nabla^{(2)}\log\left[\prod_{i=1}^n \frac{\exp\left(Z_{ij}\left(u_j+\sum_{k\neq j} u_kZ_{ik}\right)\right)}{1+\exp\left(u_j+\sum_{k\neq j} u_kZ_{ik}\right)}\right] = -\sum_{i=1}^ng^{(2)}\left(\pscal{u}{Z_{i(j)}}\right)Z_{i(j)} Z_{i(j)}',\]
where $Z_{i(j)}=(Z_{i1},\ldots,Z_{i,j-1},1,Z_{i,j+1},\ldots,Z_{ip})$. Hence $-n\mathcal{H}^{(j)}$ is the Fisher information matrix in the conditional model that regress the $j$-th column of $Z$ on the remaining. The quantities $\underline{\kappa}_2(s)$ and $\underline{\kappa}_2$ are (the minimum over $j$ of) restricted smallest eigenvalues of these information matrices.  We will work under the assumption that $\underline{\kappa}_2(s)>0$ and $\underline{\kappa}_2>0$, for some well-chosen $s$. Similar assumptions are made in most work on high-dimensional discrete graphical models (\cite{ravikumaretal10,atchade:ejs14,barber:drton:2015}). Although these assumptions are very natural in this context, to the best of our knowledge there does not seem to exist any easy way of checking them for a given parameter value $\theta_\star$.
%Admittedly, these assumptions are not easy to check in practice. But to our defense, we note that statistical inference for models with information singularity is still largely an under-developed topic, and it would be a daunting task to tackle this issue in the present context.
\end{remark}

We will take the prior parameter $\rho$ as 
\begin{equation}\label{rho:ising}
\rho = 24\sqrt{n\log(p)}.
\end{equation}

In order to apply Theorem \ref{thm2}  we view $\rset^{p\times p}$ as $\rset^{d}$, with $d=p^2$, equipped with the Frobenius norm $\normfro{\theta}\eqdef \sqrt{\textsf{Tr}(\theta'\theta)}$, and inner product $\pscal{\theta}{\vartheta}_\textsf{F} \eqdef \textsf{Tr}(\theta'\vartheta)$, where $\textsf{Tr}(\theta)$ denotes the trace of the matrix $\theta$. Throughout this section, the norm $\|\cdot\|_2$ always denotes the Euclidean norm on $\rset^p$. We will work with split cones of the form $\{\theta\in\rset^{p\times p}:\; \|\theta_{\cdot j}\|_0\leq s_j,\;1\leq j\leq p\}$. 

\begin{theorem}\label{thm:Ising:1}
Consider the quasi-posterior distribution (\ref{quasipost:ising}). Suppose that C\ref{C1} holds, the prior $\{\pi_\delta,\;\delta\in\Delta_p\}$ satisfies H\ref{H2} (with $d$ replaced by $p$), and $\rho$ is given by (\ref{rho:ising}). For $1\leq j\leq p$ set 
\begin{equation}\label{zeta:ising}
\zeta_j \eqdef s_{\star j} +\frac{4}{c_4} +\frac{2}{c_4}\left(1+\frac{128}{\underline{\kappa}_2} + \frac{s_{\star j}}{64(\log(p))^2} +\frac{\log(4 e)}{\log(p)}\right)s_{\star j},\end{equation}
$\bar s_j \eqdef \lceil s_{\star j}+\zeta_j\rceil$, and $\bar s\eqdef \max_{1\leq j\leq p}\bar s_j$. If $\underline{\kappa}\eqdef \min(\underline{\kappa}_2,\underline{\kappa}_2(\bar s))>0$, then there exist universal finite positive constants $A_1, A_2$ such that for all $p$ large enough and 
\begin{equation}\label{bound:np:ising:1} 
n\geq A_1\left(\frac{1}{\underline{\kappa}}\sum_{j=1}^p\bar s_j\right)^2\log(p),\end{equation}
the following statements hold.
\begin{enumerate}
\item \[\PE^{(n)}\left[\check\Pi_{n,d}\left(\left\{\theta\in\rset^{p\times p}:\;\|\theta_{\cdot j}\|_0>\zeta_j,\;\mbox{ for some } j\right\}\vert Z\right)\right]\leq e^{-A_2 n} + \frac{4}{p}.\]
\item There exists a finite constant $M_0>2$ (that depends on the constants in H\ref{H2}), such that
\begin{multline*}
\PE^{(n)}\left[\check\Pi_{n,d}\left(\left\{\theta\in\rset^{d\times d}:\;\normfro{\theta-\theta_\star}> \frac{M_0}{\underline{\kappa}_2(\bar s)}\sqrt{\left(\sum_{j=1}^p\bar s_j\right)\frac{\log(p)}{n}}\right\}\vert Z\right)\right] \\
 \leq 2e^{-A_2 n} + \frac{12}{p}.\end{multline*}
\end{enumerate}
\end{theorem}
\begin{proof}
See Section \ref{sec:proof:thm:Ising}.
\end{proof}

\medskip
If $p$ and $n$ are large while $\underline{\kappa}$ remains bounded away from zero, Theorem \ref{thm:Ising:1}-Part(1) implies that the quasi-posterior distribution $\check\Pi_{n,d}$ puts high probability on matrices of $\rset^{d\times d}$ with the same sparsity pattern as  $\theta_\star$, and Theorem \ref{thm:Ising:1}-Part(2) implies that in this case, the rate of convergence in the Frobenius norm is of order
\[O\left(\sqrt{\frac{(p+S)\log(p)}{n}}\right),\]
where $S\eqdef \sum_{j=1}^p s_{\star j}$ is twice the number of non-zero components of $\theta_\star$.  As we show next, faster convergence rate is possible if one is only interested in components of $\theta$. We consider the norm
\[\tnorm{\theta} \eqdef \max_{1\leq j\leq p}\|\theta_{\cdot j}\|_2,\;\;\theta\in\rset^{p\times p}.\]

\begin{theorem}\label{thm:Ising:3}
Under the assumptions of Theorem \ref{thm:Ising:1}, if $\underline{\kappa}>0$, then there exist finite universal constants $A_1,A_2$, and a finite constant  $M_0>2$ (that depends only on the constants in H\ref{H2}) such that for all $p$ large enough, and for
\[ n\geq A_1 \left(\frac{\bar s}{\underline{\kappa}(\bar s)}\right)^2\log(p),\]
\[\PE^{(n)}\left[\check\Pi_{n,d}\left(\left\{\theta\in\rset^{d\times d}:\;\tnorm{\theta-\theta_\star}> \frac{M_0}{\underline{\kappa}_2(\bar s)}\sqrt{\frac{\bar s\log(p)}{n}}\right\}\vert Z\right)\right] \leq 2e^{-A_2 n} + \frac{12}{p}.\]
\end{theorem}
\begin{proof}
See Section \ref{sec:proof:thm:Ising}.
\end{proof}

\section{Proofs}\label{sec:proofs}
\subsection{Proof of Theorem \ref{thm2}}\label{sec:proof:thm2}
To improve readability we split the proof in three parts. The first part deals with the normalizing constant of the quasi-posterior distribution, the second part deals with the existence of test functions, and the proof of the theorem itself is given in the third part.

\subsubsection{On the normalizing constant of the quasi-posterior distribution}
The next lemma provides a lower bound on the normalizing constant of the quasi-posterior distribution (\ref{quasi:post}), following an approach initially developed by \cite{castillo:etal:14}.
 
\begin{lemma}\label{lem1:thm2}
Assume H\ref{H0}-H\ref{H1}. Fix $L\geq 0$, and a split cone $\Theta\supseteq\Theta_\star$. For all $z\in\hat\e_{n,1}(\Theta,L)$,
\begin{equation}\label{lem1:thm2:bound1}
\int_{\rset^d}\frac{q_{n,\theta}(z)}{q_{n,\theta_\star}(z)}\Pi(\rmd \theta) \geq   \pi_{\delta_\star} \left(\frac{\rho^2}{L+\rho^2}\right)^{s_\star} e^{-\rho\|\theta_\star\|_1}.\end{equation}
\end{lemma}
\begin{proof}
Using the definition of the prior $\Pi$, we have 
\begin{equation}\label{eq1:proof:lem1:thm2}
\int_{\rset^d}\frac{q_{n,\theta}(z)}{q_{n,\theta_\star}(z)}\Pi(\rmd \theta)
\geq  \pi_{\delta_\star}\left(\frac{\rho}{2}\right)^{s_\star}\int_{\theta_\star + \Theta_\star}\frac{q_{n,\theta}(z)}{q_{n,\theta_\star}(z)} e^{-\rho\|\theta\|_1}\mu_{d,\delta_\star}(\rmd \theta).\end{equation}
For $z\in \hat\e_{n,1}(\Theta,L)$, and $\theta\in\theta_\star + \Theta_\star\subseteq \theta_\star + \Theta$, 
\[\log q_{n,\theta}(z) -\log q_{n,\theta_\star}(z) \geq \pscal{\nabla\log q_{n,\theta_\star}(z)}{\theta-\theta_\star} -\frac{L}{2}\|\theta-\theta_\star\|_2^2.\]
Setting $\vartheta = \nabla\log q_{n,\theta_\star}(z)$, (\ref{eq1:proof:lem1:thm2}) then gives
\begin{eqnarray}\label{eq2:proof:lem1:thm2}
\int_{\rset^d}\frac{q_{n,\theta}(z)}{q_{n,\theta_\star}(z)}\Pi(\rmd \theta)  &\geq &  \pi_{\delta_\star}\left(\frac{\rho}{2}\right)^{s_\star} e^{-\rho\|\theta_\star\|_1}\\
&& \times \int_{\theta_\star + \Theta_\star}e^{\pscal{\vartheta}{\theta-\theta_\star}-\frac{L}{2}\|\theta-\theta_\star\|_2^2 } e^{-\rho\|\theta-\theta_\star\|_1}\mu_{d,\delta_\star}(\rmd \theta)\nonumber.
\end{eqnarray}
We note that the support of the measure $\mu_{d,\theta_\star}$ is $\Theta_\star=\theta_\star+\Theta_\star$. Using this and the change of variable $\theta=\theta_\star+z$, we see that  the integral on the right-hand size of (\ref{eq2:proof:lem1:thm2}) is 
\[\int_{\rset^d} e^{\pscal{\vartheta}{z} -\frac{L}{2}\|z\|_2^2 -\rho\|z\|_1}\mu_{d,\delta_\star}(\rmd z).\]
By Jensen's inequality, 
\begin{multline*}
\int_{\rset^d} e^{\pscal{\vartheta}{z}}\frac{e^{-\frac{L}{2}\|z\|_2^2 -\rho\|z\|_1}}{\int_{\rset^d}e^{-\frac{L}{2}\|u\|_2^2 -\rho\|u\|_1}\mu_{d,\delta_\star}(\rmd u)}\mu_{d,\delta_\star}(\rmd z)\\
\geq 
\exp\left(\int_\rset \pscal{\vartheta}{z} \frac{e^{-\frac{L}{2}\|z\|_2^2 -\rho\|z\|_1}}{\int_{\rset^d}e^{-\frac{L}{2}\|u\|_2^2 -\rho\|u\|_1}\mu_{d,\delta_\star}(\rmd u)}\mu_{d,\delta_\star}(\rmd z)\right)=1.\end{multline*}
Using this, and going back to (\ref{eq1:proof:lem1:thm2}) we conclude that
\[\int_{\rset^d}\frac{q_{n,\theta}(z)}{q_{n,\theta_\star}(z)}\Pi(\rmd \theta) \geq   \pi_{\delta_\star}\left(\frac{\rho}{2}\right)^{s_\star} e^{-\rho\|\theta_\star\|_1} \int_{\rset^d}e^{-\frac{L}{2}\|u\|_2^2 -\rho\|u\|_1}\mu_{d,\delta_\star}(\rmd u).\]
Now, note that 
\[\int_{\rset^d}e^{-\frac{L}{2}\|u\|_2^2 -\rho\|u\|_1}\mu_{d,\delta_\star}(\rmd u)=\left(\int_\rset e^{-\rho|z| -\frac{L}{2}z^2}\rmd z\right)^{s_\star}.\]
It is easy to calculate that  for $a\geq 0, b>0$
\begin{equation}\label{en:nc}
\int_{\rset}e^{-\frac{a}{2}u^2 -b|u|}\rmd u = \frac{2}{\sqrt{a}}\frac{1-\Phi\left(\frac{b}{\sqrt{a}}\right)}{\phi\left(\frac{b}{\sqrt{a}}\right)},\end{equation}
where $\phi$ is the density of the standard normal distribution, and $\Phi$ its cdf. The formula continues to hold by continuity at $a=0$. The ratio $(1-\Phi(z))/\phi(z)$ (known as Mills' ratio), satisfies 
\begin{equation}\label{mills:ratio:ineq}
\frac{z}{1+z^2}\leq \frac{2}{z+\sqrt{z^2+4}}\leq \frac{1-\Phi(z)}{\phi(z)}\leq \frac{4}{3z+\sqrt{z^2+8}},\;\; z\geq 0,\end{equation}
see for instance \cite{baricz:2008}~Theorem 2.3 for a proof. We use this inequality and (\ref{en:nc}) to conclude that 
\[\int_\rset e^{-\rho|z|-\frac{L}{2} z^2}\rmd z\geq \frac{2\rho}{L+\rho^2},
\]
and the lemma follows easily.
\end{proof}

\subsubsection{On the existence of test functions}
In this paragraph we establish the existence of test functions to test the density $f_{n,\theta_\star}$ against some mis-specified alternatives $Q_{n,\theta}$ defined below. The result is based on Lemma 6.1 of \cite{kleijn:vaart:06}, that we shall recall first for completeness. For any two integrable non-negative functions $q_1,q_2$ on $\Xset^{(n)}$, and for $\alpha\in (0,1)$, the Hellinger transform $\H_\alpha(q_1,q_2)$ is defined as 
\[\H_\alpha(q_1,q_2)\eqdef\int_{\Xset^{(n)}} q_1^{\alpha}(z)q^{1-\alpha}_2(z)\rmd z.\]
Here we work with the case $\alpha=1/2$, and set $\H(q_1,q_2) \eqdef \H_{1/2}(q_1,q_2)$. 
\begin{lemma}[\cite{kleijn:vaart:06}~Lemma 6.1]\label{exist:test}
Let $p$ be a probability density function on $\Xset^{(n)}$ and $\mathcal{Q}$ a class of non-negative integrable functions on $\Xset^{(n)}$. Then
\begin{equation}\label{lem:klein}
\inf_{\phi}\sup_{q\in\mathcal{Q}} \left[\int_{\Xset^{(n)}} \phi(z)p(z)\rmd z + \int_{\Xset^{(n)}} (1-\phi(z))q(z)\rmd z\right] \leq \sup_{q\in\textsf{conv}(\mathcal{Q})} \H(p,q),\end{equation}
where $\textsf{conv}(\mathcal{Q})$ is the convex hull of $\mathcal{Q}$, and the infimum in (\ref{lem:klein}) is taken over all test functions, that is all measurable functions $\phi:\;\Xset^{(n)}\to [0,1]$. Furthermore, there exists a test function $\phi$ that attains the infimum.
\end{lemma}

To derive the test function for our quasi-likelihood setting, we will also need the following easy result.

\begin{lemma}\label{lem3:thm2}
Fix  $\lambda\geq 0$, a split cone $\Theta$, and a rate function $\r$ such that $\phi_\r(2\lambda)$ is finite. For any $\theta\in\theta_\star +\Theta$ such that  $\nnorm{\theta-\theta_\star}\geq \phi_\r(2\lambda)$, we have
\[\frac{q_{n,\theta}(z)}{q_{n,\theta_\star}(z)}\leq e^{- \frac{1}{4}\r(\nnorm{\theta-\theta_\star})},\;z\in\e_{n,0}(\Theta,\lambda)\cap\check\e_{n,1}(\Theta,\r).\]
\end{lemma}
\begin{proof}
For all $z\in\Xset^{(n)}$, and $\theta\in\rset^d$, we have
\[
\frac{q_{n,\theta}(z)}{q_{n,\theta_\star}(z)} = \exp\left[\pscal{\nabla\log q_{n,\theta_\star}(z)}{\theta-\theta_{\star}} +\L_{n,\theta}(z)\right].\]
By the definition of $\check\e_{n,1}(\Theta,\r)$, for $\theta\in\theta_\star +\Theta$ and $z\in \check\e_{n,1}(\Theta,\r)$, we have $\L_{n,\theta}(z) \leq -\frac{1}{2} \r(\nnorm{\theta-\theta_\star})$. And by the definition of $\e_{n,0}(\Theta,\lambda)$, for $z\in \e_{n,0}(\Theta,\lambda)$, and $\theta\in\theta_\star +\Theta$, we have 
\[\left|\pscal{\nabla\log q_{n,\theta_\star}(z)}{\theta-\theta_{\star}}\right| \leq  \frac{\lambda}{2} \nnorm{\theta-\theta_\star}.\]
Hence,  for $z\in\e_{n,0}(\Theta,\lambda)\cap\check\e_{n,1}(\Theta,\r)$, and $\theta\in\theta_\star +\Theta$,
\begin{equation}\label{eq1:lem3:thm2}
\frac{q_{n,\theta}(z)}{q_{n,\theta_\star}(z)}\leq  \exp\left[\frac{\lambda}{2}\nnorm{\theta-\theta_\star} - \frac{1}{2}\r(\nnorm{\theta-\theta_\star})\right].\end{equation}
If in addition $\nnorm{\theta-\theta_\star}\geq \phi_\r(2\lambda)$, then from the properties of the rate function $\r$, we have $2\lambda \nnorm{\theta-\theta_\star} - \r(\nnorm{\theta-\theta_\star})\leq 0$, and the result follows.
\end{proof}

Our main result on the existence of test functions follows. We recall that for $M>0$, and a split cone $\Theta$, $\textsf{B}_d(\Theta,M)\eqdef\{\theta\in\theta_\star+\Theta:\;\mbox{ s.t. }\;\nnorm{\theta-\theta_\star}\leq M\}$, and for $\epsilon>0$, $\textsf{D}(\epsilon,\textsf{B}_d(\Theta,M))$ denotes the $\epsilon$-packing number of $\textsf{B}_d(\Theta,M)$ in the norm $\nnorm{\cdot}$.

\begin{lemma}\label{lem4:thm2}
Fix $\lambda\geq 0$, a split cone $\Theta$, and a rate function $\r$ such that $\tilde\epsilon\eqdef \phi_\r(2\lambda)$ is finite. Set $\bar\e_n\eqdef \e_{n,0}(\Theta,\lambda)\cap\check\e_{n,1}(\Theta,\r)$. For $\theta\in\rset^d$, define the function
\begin{equation}\label{def:Qn}
Q_{n,\theta}(z)\eqdef\textbf{1}_{\bar\e_n}(z)\frac{q_{n,\theta}(z)}{q_{n,\theta_\star}(z)}f_{n,\theta_\star}(z),\;\;\;z\in\Xset^{(n)}.\end{equation}
For any $M>2$, there exists a measurable function $\phi:\;\Xset^{(n)} \to [0,1]$ such that,
\[\PE^{(n)}(\phi(Z))\leq \sum_{j\geq 1}\textsf{D}_je^{-\frac{1}{8}\r(\frac{jM\tilde\epsilon}{2})},\]
where $\textsf{D}_j \eqdef \textsf{D}\left(\frac{j M\tilde \epsilon}{2},\textsf{B}_d(\Theta,(j+1)M\tilde\epsilon)\right)$. Furthermore, for all $j\geq 1$, all $\theta\in\theta_\star +\Theta$ such that  $\nnorm{\theta-\theta_\star}> jM\tilde \epsilon$, 
\[\int_{\Xset^{(n)}}(1-\phi(z))Q_{n,\theta}(z)\rmd z \leq e^{-\frac{1}{8}\r(\frac{jM\tilde\epsilon}{2})}.\]
\end{lemma}
\begin{proof}
First, notice that the function $z\mapsto Q_{n,\theta}(z)$ is integrable for all $\theta\in\theta_\star+\Theta$. Indeed, using (\ref{eq1:lem3:thm2}) for any such $\theta$,  and for $z\in\bar\e_n$: $\frac{q_{n,\theta}(z)}{q_{n,\theta_\star}(z)}\leq \exp\left(\frac{\lambda}{2}\nnorm{\theta-\theta_\star}\right)$. Hence,
\[\int_{\Xset^{(n)}}Q_{n,\theta}(z)\rmd z=\int_{\bar\e_n}\frac{q_{n,\theta}(z)}{q_{n,\theta_\star}(z)}f_{n,\theta_\star}(z)\rmd z\leq e^{\frac{\lambda}{2}\nnorm{\theta-\theta_\star}}.\]

Now, fix $\epsilon>2\tilde\epsilon$ (where $\tilde\epsilon=\phi_\r(2\lambda)$), and fix $\theta\in \theta_\star+\Theta$ such that $\nnorm{\theta-\theta_\star}>\epsilon$. Set $\mathcal{P}_\theta\eqdef \{Q_{n,u}:\; u\in\theta_\star+\Theta \mbox{ and }\; \nnorm{u-\theta}\leq \epsilon/2\}$, and let $\textsf{conv}(\mathcal{P}_\theta)$ denote the convex hull of the set $\mathcal{P}_\theta$. By Lemma \ref{exist:test} applied with $p=f_{n,\theta_\star}$, and $\mathcal{Q}=\mathcal{P}_\theta$, there exists a measurable function  $\phi_{\theta}:\;\Xset^{(n)}\to [0,1]$ such that
\begin{multline}\label{eq1:lem4:thm2}
\PE^{(n)}\left[\phi_{\theta}(Z)\right]\leq \sup_{Q\in\textsf{conv}(\mathcal{P}_\theta)}\H(f_{n,\theta_\star},Q)\\
\;\;\;\mbox{ and }\;\;\; \sup_{Q\in\mathcal{P}_\theta}\int_{\Xset^{(n)}} (1-\phi_{\theta}(z))Q(z)\rmd z\leq \sup_{Q\in\textsf{conv}(\mathcal{P}_\theta)}\H(f_{n,\theta_\star},Q).
\end{multline}
Any $Q\in\textsf{conv}(\mathcal{P}_\theta)$ can be written as a finite convex combination $Q=\sum_j \alpha_j Q_{n,u_j}$ where $\alpha_j\geq 0$, $\sum_j\alpha_j=1$, $u\in \theta_\star+\Theta$, and $\nnorm{u_j-\theta}\leq\epsilon/2$. However, since $\nnorm{\theta-\theta_\star}>\epsilon$, and $\nnorm{u_j-\theta}\leq \epsilon/2$, we see that $\nnorm{u_j-\theta_\star}>\epsilon/2>\tilde\epsilon$. Hence, using Lemma \ref{lem3:thm2} and the definition of the Hellinger transform, we have
\[\H(f_{n,\theta_\star},Q)=\int_{\Xset^{(n)}}\sqrt{\sum_j\alpha_j\textsf{1}_{\bar\e_n}(z)\frac{q_{n,u_j}(z)}{q_{n,\theta_\star}(z)}}f_{n,\theta_\star}(z)\rmd z\leq \sqrt{\sum_j\alpha_j e^{-\frac{1}{4} \r(\nnorm{u_j-\theta_\star})}}.\]
 Hence (\ref{eq1:lem4:thm2}) becomes
\begin{multline}\label{eq2:lem4:thm2}
\PE^{(n)}\left[\phi_{\theta}(Z)\right]\leq e^{-\frac{1}{8} \r(\frac{\epsilon}{2})}\;\;\;\mbox{ and }\;\;\; \sup_{Q\in\mathcal{P}_\theta}\int_{\Xset^{(n)}} (1-\phi_{\theta}(z))Q(z)\rmd z\leq e^{-\frac{1}{8} \r(\frac{\epsilon}{2})}.
\end{multline}

Now, given $M>2$, we write $\{\theta\in\theta_\star + \Theta:\; \|\theta-\theta_\star\|_2>  M\tilde\epsilon\}=\cup_{j\geq 1}\textsf{B}(j)$, where
\[\textsf{B}(j)=\{\theta\in\theta_\star + \Theta, \mbox{ s.t. }\; jM\tilde\epsilon<\nnorm{\theta-\theta_\star}\leq (j+1)M\tilde \epsilon\}.\]
For each $j\geq 1$, let $\Sa_{j}$ be a maximal $(jM\tilde\epsilon/2)$-separated points in $\textsf{B}(j)$. For each $j$ for which $\textsf{B}(j)\neq \emptyset$, and each point $\theta_k\in\Sa_{j}$ we can construct a test function $\phi_{\theta_k}$ as above, with $\epsilon = jM\tilde\epsilon$. Then we set 
\[\phi=\sup_{j\geq 1}\max_{\theta_k\in \Sa_{j}} \phi_{\theta_k},\]
where the supremum in $j$ is over the indexes for which $\textsf{B}(j)\neq \emptyset$. Now, any $\theta\in \theta_\star +\Theta$ such that  $\nnorm{\theta-\theta_\star}>jM\tilde\epsilon$ will be within $iM\tilde \epsilon/2$ of a point $\theta_k$ in $\Sa_{i}$ for some $i\geq j$. Hence by (\ref{eq2:lem4:thm2}), for any such $\theta$,
\[\int_{\Xset^{(n)}} (1-\phi(z))Q_{n,\theta}(z)\rmd z\leq \int_{\Xset^{(n)}} (1-\phi_{\theta_k}(z))Q_{n,\theta}(z)\rmd z\leq e^{-\frac{1}{8}\r(\frac{jM\tilde\epsilon}{2})}.\]
Notice that the size of $\Sa_j$ is upper bounded by $\textsf{D}_j$. Using this and (\ref{eq2:lem4:thm2}), we get
\[\PE^{(n)}\left[\phi(Z)\right]\leq\sum_{j\geq 1}\textsf{D}_je^{-\frac{1}{8}\r(\frac{jM\tilde\epsilon}{2})},\]
which proves the lemma.
\end{proof}

\subsubsection{Proof of Theorem \ref{thm2}-Part(1)}\label{sec:proof:thm2-1}
For integer $k\geq 0$, let $\A_k\eqdef \{\theta\in\rset^d:\;\|\theta\|_0\geq s_\star +k\}$. We have
\[\PE^{(n)}\left(\check\Pi_{n,d}(\A_k\vert Z)\right) \leq  \PP^{(n)}\left(Z\notin \e_n\right) + T,\]
where  $T = \PE^{(n)}\left[\textbf{1}_{\e_n}(Z)\frac{\int_{\A_k}\frac{q_{n,\theta}(Z)}{q_{n,\theta_\star}(Z)}\Pi(\rmd\theta)}{\int_{\rset^d}\frac{q_{n,\theta}(Z)}{q_{n,\theta_\star}(Z)}\Pi(\rmd\theta)}\right]$. We use Lemma \ref{lem1:thm2}, and Fubini's theorem to write
\begin{eqnarray}\label{eq:thm1:controlT}
T & \leq &  \frac{1}{\pi_{\delta_\star}}\left(1+\frac{\bar L}{\rho^2}\right)^{s_\star} e^{\rho\|\theta_\star\|_1} \PE^{(n)}\left[\textbf{1}_{\e_n}(Z)\int_{\A_k}\frac{q_{n,\theta}(Z)}{q_{n,\theta_\star}(Z)}\Pi(\rmd\theta)\right]\nonumber\\
&= &  \frac{1}{\pi_{\delta_\star}}\left(1+\frac{\bar L}{\rho^2}\right)^{s_\star}\nonumber\\
&&\times \sum_{\delta\in\Delta_d}\pi_\delta\left(\frac{\rho}{2}\right)^{\|\delta\|_0} \int_{\A_k} \PE^{(n)}\left[\textbf{1}_{\e_n}(Z)\frac{q_{n,\theta}(Z)}{q_{n,\theta_\star}(Z)}\frac{e^{-\rho\|\theta\|_1}}{e^{-\rho\|\theta_\star\|_1}}\right]\mu_{d,\delta}(\rmd \theta).
\end{eqnarray}
We need to control the expectation on the right-side of (\ref{eq:thm1:controlT}). First note that $\e_{n,0}(\rset^d,\rho)=\{z\in\Xset^{(n)}:\; \|\nabla\log q_{n,\theta_\star}(z)\|_\infty\leq \frac{\rho}{2}\}$. With this in mind, we see that $z\in\e_n\subseteq \e_{n,0}(\rset^d,\rho)$, and $\theta\in\rset^d$, we have
\begin{eqnarray*}
\frac{q_{n,\theta}(z)}{q_{n,\theta_\star}(z)} &=& \exp\left[\pscal{\nabla\log q_{n,\theta_\star}(z)}{\theta-\theta_{\star}} +\L_{n,\theta}(z)\right],\\
&\leq & \exp\left[ \frac{\rho}{2}\|\theta-\theta_\star\|_1 + \L_{n,\theta}(z)\right].
\end{eqnarray*}
Setting $B(\theta) \eqdef \frac{\rho}{2}\|\theta-\theta_\star\|_1 + \rho(\|\theta_\star\|_1 -\|\theta\|_1)$, it follows that for all $\theta\in\rset^d$,
\begin{equation}\label{eq:thm1:control:exp}
\PE^{(n)}\left[\textbf{1}_{\e_n}(Z)\frac{q_{n,\theta}(Z)}{q_{n,\theta_\star}(Z)}\frac{e^{-\rho\|\theta\|_1}}{e^{-\rho\|\theta_\star\|_1}}\right] \leq e^{B(\theta)}
\PE^{(n)}\left[\textbf{1}_{\e_n}(Z)\exp\left( \L_{n,\theta}(Z)\right)\right].
\end{equation}
We then write
\begin{eqnarray}\label{eq:thm1:control:th_star}
\|\theta_\star\|_1 + \frac{1}{2}\|\theta-\theta_\star\|_1 & =  & \|\theta_\star\|_1 + \frac{1}{2}\|\theta\cdot\delta_\star^c\|_1 +\frac{1}{2}\|(\theta-\theta_\star)\cdot\delta_\star\|_1\nonumber\\
&\leq & \|\theta\|_1 -\frac{1}{2} \|\theta\cdot\delta_\star^c\|_1 + \frac{3}{2}\|(\theta-\theta_\star)\cdot\delta_\star\|_1.
\end{eqnarray}
Using this bound in the expression of $B(\theta)$ shows that if $\theta\notin\theta_\star + \N$, then we have
\begin{eqnarray}\label{eq:thm1:control:th_star:2}
B(\theta) &\leq & -\frac{\rho}{2}\|\theta\cdot\delta_\star^c\|_1 + \frac{3\rho}{2} \|(\theta-\theta_\star)\cdot\delta_\star\|_1 \\
&\leq & -\frac{\rho}{4}\|\theta-\theta_\star\|_1.\nonumber
\end{eqnarray}
This bound together with the fact that the expectation on the right-side of (\ref{eq:thm1:control:exp}) is always smaller or equal to $1$ (which follows from the concaveness assumption)  show that when $\theta\notin \theta_\star + \N$, 
\[
\PE^{(n)}\left[\textbf{1}_{\e_n}(Z)\frac{q_{n,\theta}(Z)}{q_{n,\theta_\star}(Z)}\frac{e^{-\rho\|\theta\|_1}}{e^{-\rho\|\theta_\star\|_1}}\right]
\leq e^{-\frac{\rho}{4}\|\theta-\theta_\star\|_1}.\]
Now, consider the case where $\N\neq\emptyset$, and $\theta-\theta_\star \in \N$. In that case, the definition of the set $\check\e_{n,1}(\N,\r)$ and  (\ref{eq:thm1:control:exp}) yield
\[
\PE^{(n)}\left[\textbf{1}_{\e_n}(Z)\frac{q_{n,\theta}(Z)}{q_{n,\theta_\star}(Z)}\frac{e^{-\rho\|\theta\|_1}}{e^{-\rho\|\theta_\star\|_1}}\right]\leq e^{B(\theta) - \frac{1}{2}\r(\|\theta-\theta_\star\|_2)}.\]
 From (\ref{eq:thm1:control:th_star:2}), 
\[B(\theta) - \frac{1}{2}\r(\|\theta-\theta_\star\|_2) \leq  -\frac{\rho}{2}\|\theta-\theta_\star\|_1 + 2\rho \|(\theta-\theta_\star)\cdot\delta_\star\|_1  - \frac{1}{2}\r(\|\theta-\theta_\star\|_2),\]
and
\begin{eqnarray*}
2 \rho \|(\theta-\theta_\star)\cdot\delta_\star\|_1  - \frac{1}{2}\r(\|\theta-\theta_\star\|_2) &\leq& 2\rho\sqrt{s_\star}\|\theta-\theta_\star\|_2 - \frac{1}{2}\r(\|\theta-\theta_\star\|_2),\\
& \leq & -\frac{1}{2}\left[\r(\|\theta-\theta_\star\|_2) - 4\rho\sqrt{s_\star}\|\theta-\theta_\star\|_2\right]\\
& \leq&  -\frac{1}{2} \inf_{x> 0} \left[\r(x)-4\rho s_\star^{1/2} x\right].\end{eqnarray*}
Therefore, when $\theta\neq \theta_\star$, and $\theta\in\theta_\star + \N$, we have
 \[\PE^{(n)}\left[\textbf{1}_{\e_n}(Z)\frac{q_{n,\theta}(Z)}{q_{n,\theta_\star}(Z)}\frac{e^{-\rho\|\theta\|_1}}{e^{-\rho\|\theta_\star\|_1}}\right] \leq e^{\textsf{a}} e^{-\frac{\rho}{2}\|\theta-\theta_\star\|_1},\]
where $\textsf{a} = -\frac{1}{2} \inf_{x> 0} \left[\r(x)-4\rho s_\star^{1/2} x\right]$. Note that $\textsf{a}>0$, since $\lim_{x\downarrow 0}\r(x)/x=0$. In view of these calculations and (\ref{eq:thm1:controlT}), we conclude that 
\[T \leq   e^{\textsf{a}}\left(1+\frac{\bar L}{\rho^2}\right)^{s_\star}\frac{1}{\pi_{\delta_\star}} \sum_{\delta\in\Delta_d}\pi_\delta\left(\frac{\rho}{2}\right)^{\|\delta\|_0} \int_{\A_k}e^{-\frac{\rho}{4}\|\theta-\theta_\star\|_1}\mu_{d,\delta}(\rmd \theta).\]
Note that $\mu_{d,\delta}(\A_k)=0$ if $\|\delta\|_0<s_\star+k$, and 
\[\left(\frac{\rho}{2}\right)^{\|\delta\|_0}\int_{\rset^d} e^{ -\frac{\rho}{4}\|\theta-\theta_\star\|_1}\mu_{d,\delta}(\rmd \theta)\leq \left(\frac{\rho}{2}\right)^{\|\delta\|_0}\left(\int_\rset e^{-\frac{\rho}{4}|z|}\rmd z\right)^{\|\delta\|_0} = 4^{\|\delta\|_0}.\]
Therefore,
\[T \leq   e^{\textsf{a}}\left(1+\frac{\bar L}{\rho^2}\right)^{s_\star}\frac{1}{\pi_{\delta_\star}} \sum_{\delta:\;\|\delta\|_0\geq s_\star +k}\pi_\delta 4^{\|\delta\|_0}.\]
Using H\ref{H2},
\begin{multline*}
\frac{1}{\pi_{\delta_\star}}\sum_{\delta:\;\|\delta\|_0\geq s_\star+k}\pi_\delta 4^{\|\delta\|_0} = \frac{{d\choose s_\star}}{g_{s_\star}} \sum_{j=s_\star+k}^d 4^j g_j 
\leq  \frac{{d\choose s_\star}}{g_{s_\star}} \sum_{j=s_\star+k}^d 4^j \left(\frac{c_2}{d^{c_4}}\right)^{j-s_\star}g_{s_\star} \\
= {d\choose s_\star} 4^{s_\star} \sum_{j=s_\star+k}^d \left(\frac{4c_2}{d^{c_4}}\right)^{j-s_\star}.\end{multline*}
For $d$ large enough so that $\frac{4c_2}{d^{c_4}}<1$, we have $\sum_{j=s_\star+k}^d \left(\frac{4c_2}{d^{c_4}}\right)^{j-s_\star}\leq 2\left(\frac{4c_2}{d^{c_4}}\right)^{k}$, which proves the stated bound.
\begin{flushright}
$\square$
\end{flushright}
%\end{proof}

\subsubsection{Proof of Theorem \ref{thm2}-Part(2)}
Define $U(\bar\epsilon)\eqdef\{\theta\in\theta_\star +\bar\Theta:\; \nnorm{\theta-\theta_\star}>M_0\bar\epsilon\}$. We apply  Lemma \ref{lem4:thm2}  with $\lambda=\bar\lambda$, $\Theta=\bar\Theta$,  the rate function $\r$ and with $M=M_0>2$. Notice $\bar\epsilon = \phi_\r(2\bar\lambda)$ is called $\tilde\epsilon$ in Lemma \ref{lem4:thm2}. By Lemma \ref{lem4:thm2} there exists a measurable functions $\phi:\;\Xset^{(n)}\to [0,1]$ such that 
\begin{equation}\label{bound2:proof:thm2}
\PE^{(n)}\left[\phi(Z)\right]\leq \sum_{j\geq 1}\textsf{D}_je^{-\frac{1}{8}\r(\frac{jM_0\bar\epsilon}{2})},\end{equation}
where $\textsf{D}_j \eqdef \textsf{D}\left(\frac{j M_0\bar \epsilon}{2},\textsf{B}_d(\bar\Theta,(j+1)M_0\bar\epsilon)\right)$.  Using the test function $\phi$, we have
\[\check\Pi_{n,d}\left(U(\bar\epsilon)\vert Z\right) \leq  \phi(Z) + (1-\phi(Z))\check\Pi_{n,d}\left(U(\bar\epsilon)\vert Z\right).\]
In view of (\ref{bound2:proof:thm2}), it remains only to control the expectation of $(1-\phi(Z))\check\Pi_{n,d}\left(U(\bar\epsilon)\vert Z\right)$. To do so, we set $\bar\e_n\eqdef \e_{n,0}(\bar\Theta,\bar \lambda)\cap\check\e_{n,1}(\bar \Theta,\r)$, so  that $\e_n\subseteq \bar \e_{n}\cap\hat\e_{n,1}(\bar \Theta,\bar L)$, and use Lemma \ref{lem1:thm2} and Fubini's theorem to write
\begin{multline}\label{bound3:proof:thm2}
\PE^{(n)}\left[(1-\phi(Z))\check\Pi_{n,d}\left(U(\bar\epsilon)\vert Z\right)\right]
=\PE^{(n)}\left[(1-\phi(Z))\frac{\int_{U(\bar\epsilon)}\frac{q_{n,\theta}(Z)}{q_{n,\theta_\star}(Z)}\Pi(\rmd \theta)}{\int \frac{q_{n,\theta}(Z)}{q_{n,\theta_\star}(Z)}\Pi(\rmd \theta)}\right]\\
\leq  \PP^{(n)}\left(Z\notin \e_n\right)  + \frac{1}{\pi_{\delta_\star}}\left(1+\frac{\rho^2}{\bar L}\right)^{s_\star}e^{\rho\|\theta_\star\|_1}\\ \times \int_{U(\bar\epsilon)}\PE^{(n)}\left[\textbf{1}_{\bar \e_n}(Z) (1-\phi(Z))\frac{q_{n,\theta}(Z)}{q_{n,\theta_\star}(Z)}\right]\Pi(\rmd \theta).\end{multline}
 We split $U(\bar\epsilon)$ as $U(\bar\epsilon)=\cup_{j\geq 1} \textsf{B}(j)$, where 
\[\textsf{B}(j)=\{\theta\in\theta_\star +\bar\Theta \mbox{ s.t. }\; j M_0 \bar\epsilon<\nnorm{\theta-\theta_\star}\leq (1+j) M_0 \bar\epsilon\}.\]
Therefore, and using the notation of Lemma \ref{lem4:thm2}, the integral in (\ref{bound3:proof:thm2}) is
\begin{multline*}
\int_{U_1(\bar\epsilon)}\PE^{(n)}\left[\textbf{1}_{\bar\e_n}(Z) (1-\phi(Z))\frac{q_{n,\theta}(Z)}{q_{n,\theta_\star}(Z)}\right]\Pi(\rmd \theta) \\
= \sum_{j\geq 1} \int_{\textsf{B}(j)}\left [\int_{\Xset^{(n)}} (1-\phi(z))Q_{n,\theta}(z)\rmd z\right] \Pi(\rmd \theta)\leq \sum_{j\geq 1}e^{-\frac{1}{8}\r(\frac{j M_0 \bar\epsilon}{2})}\Pi(\textsf{B}(j)).
\end{multline*}

From the prior $\Pi$, we  have
\[e^{\rho\|\theta_\star\|_1}\Pi(\textsf{B}(j)) = \sum_{\delta\in\Delta_d}\pi_\delta\left(\frac{\rho}{2}\right)^{\|\delta\|_0} \int_{\textsf{B}(j)} e^{\rho(\|\theta_\star\|_1-\|\theta\|_1)}\mu_{d,\delta}(\rmd\theta).\]
and for $\theta\in \textsf{B}(j)$,
\begin{multline*}
\rho(\|\theta_\star\|_1-\|\theta\|_1) \leq \rho\|\theta-\theta_\star\|_1 \leq -\frac{\rho}{2}\|\theta-\theta_\star\|_1 +\frac{3}{2}\rho\|\theta-\theta_\star\|_1 \\
\leq -\frac{\rho}{2}\|\theta-\theta_\star\|_1 +\frac{3}{2}\rho c_0\nnorm{\theta-\theta_\star}\leq -\frac{\rho}{2}\|\theta-\theta_\star\|_1 +3\rho c_0 jM_0\bar\epsilon\end{multline*}
where $c_0 =\sup_{u\in\bar\Theta}\sup_{v\in\bar \Theta,\;\nnorm{v}=1} |\pscal{\textsf{sign}(u)}{v}|$.  Hence
\begin{eqnarray*}
e^{\rho\|\theta_\star\|_1}\Pi(\textsf{B}(j)) & \leq & e^{3\rho c_0 jM_0\bar\epsilon}\sum_{\delta\in\Delta_d}\pi_\delta\left(\frac{\rho}{2}\right)^{\|\delta\|_0} \int_{\textsf{B}(j)} e^{-\frac{\rho}{2}\|\theta - \theta_\star\|_1}\mu_{d,\delta}(\rmd\theta),\\
& \leq & e^{3\rho c_0 jM_0\bar\epsilon}\sum_{\delta\in\Delta_d}\pi_\delta\left(\frac{\rho}{2}\right)^{\|\delta\|_0} \left(\int_{\rset} e^{-\frac{\rho}{2}|z|}\rmd z\right)^{\|\delta\|_0},\\
& = & e^{3\rho c_0 jM_0\bar\epsilon}\sum_{\delta\in\Delta_d}\pi_\delta 2^{\|\delta\|_0}.\end{eqnarray*}
Therefore, the second term on the right-hand side of (\ref{bound3:proof:thm2}) is upper bounded by
\[\frac{1}{\pi_{\delta_\star}}\left(\sum_{\delta\in\Delta_d}\pi_\delta 2^{\|\delta\|_0}\right) \left(1+\frac{\rho^2}{\bar L}\right)^{s_\star}  \sum_{k\geq 1}e^{-\frac{1}{8}\r(\frac{k M_0 \bar\epsilon}{2})} e^{3\rho c_0 k M_0\bar\epsilon}.\]
As in Part(1), using H\ref{H2} and for $d^{c_4}\geq 4c_2$,
\begin{multline*}
\pi_{\delta_\star}^{-1}\sum_{\delta\in\Delta_d}\pi_\delta 2^{\|\delta\|_0} = \frac{{d\choose s_\star}}{g_{s_\star}} \sum_{j=0}^d 2^jg_j \leq \frac{{d\choose s_\star}}{g_{s_\star}} g_0\sum_{j=0}^d\left(\frac{2c_2}{d^{c_4}}\right)^j\\
\leq 2 {d\choose s_\star} \frac{g_0}{g_{s_\star}} \leq 2  {d\choose s_\star}\left(\frac{d^{c_3}}{c_1}\right)^{s_\star}.\end{multline*}
This ends the proof. 

\begin{flushright}
$\square$
\end{flushright}

\subsection{Proof of Theorem \ref{thm1:logit}}\label{sec:proof:thm:logit}
Clearly, B\ref{B1} implies H\ref{H0}, and H\ref{H1} trivially holds true. Furthermore in this case the function $\L_{n,\theta}$ is given by
\[\L_{n,\theta}(z) = -\sum_{i=1}^n g(\pscal{x_i}{\theta}) -g(\pscal{x_i}{\theta_\star}) -g'(\pscal{x_i}{\theta_\star})\pscal{x_i}{\theta-\theta_\star},\]
which does not depend on $z$. To control this term, we will rely on a nice self-concordant properties of the logistic function $g(x)= \log(1+e^x)$ developed by \cite{bach10}~Lemma 1, which states that for all $x_0,u\in\rset$,
%\begin{equation}\label{self:conc:1}
%u^2g^{(2)}(x_0)\left(\frac{1-e^{-|u|t}}{|u|}\right)\leq \left[g'(x_0+tu)-g'(x_0)\right]u\leq %u^2g^{(2)}(x_0)\left(\frac{e^{|u|t}-1}{|u|}\right).\end{equation}
\begin{multline}\label{self:conc:2}
g^{(2)}(x_0)\left(e^{-|u|}+|u|-1\right)\leq g(x_0+u)-g(x_0)-g'(x_0)u\\
\leq g^{(2)}(x_0)\left(e^{|u|}-|u|-1\right).\end{multline}

\begin{proof}[Proof of Part(1)]
We shall apply Theorem \ref{thm2}-(Part 1). Clearly, $\theta\mapsto \log q_{n,\theta}(z)$ is concave for all $z\in\{0,1\}^n$. We define $H(x) \eqdef e^{-x}+x-1$. It can be checked that $H$ satisfies
\begin{equation}\label{prop:H}
H(x) \geq \frac{x^2}{2+x},\;\;x\geq 0.\end{equation}
This holds because $(2+x)H(x) -x^2 = (2+x)e^{-x} +x-2$, the derivative of which is $1-\frac{x+1}{e^x}\geq 0$, for all $x\geq 0$. Using (\ref{self:conc:2}), we get 
\[\L_{n,\theta}(z) \leq  -\sum_{i=1}^n g^{(2)}\left(\pscal{x_i}{\theta_\star}\right)H\left(|\pscal{x_i}{\theta-\theta_\star}|\right).\]
Furthermore, for $\theta-\theta_\star\in\N$, we have
\[\left|\pscal{x_i}{\theta-\theta_\star}\right| \leq \|X\|_\infty\|\theta-\theta_\star\|_1\leq 8\|X\|_\infty s_\star^{1/2} \|\theta-\theta_\star\|_2.\]
Using this, (\ref{prop:H}), and the definition of $\underline{\kappa}_1$, we get for all $z\in\{0,1\}^n$,
\begin{eqnarray}\label{upper:bound:L}
\L_{n,\theta}(z)&\leq &- \frac{n}{2+\max_{i}|\pscal{x_i}{\theta-\theta_\star}|}(\theta-\theta_\star)'\frac{X'WX}{n}(\theta-\theta_\star)\nonumber\\
&\leq&  - \frac{n\underline{\kappa}_1\|\theta-\theta_\star\|_2^2}{2+8\sqrt{s_\star}\|X\|_\infty\|\theta-\theta_\star\|_2},\nonumber\\
& = & -\frac{1}{2}\r(\|\theta-\theta_\star\|_2),\end{eqnarray}
where $\r(x)  =n\underline{\kappa}_1 x^2/(1+4\sqrt{s_\star}\|X\|_\infty x)$. Hence, with this particular choice of rate function $\PP^{(n)}(Z\notin \check\e_{n,1}(\N,\r)) = 0$. Since $g^{(2)}(x)\leq 1/4$, it follows that
\[\L_{n,\theta}(z) \geq -\frac{n}{8} (\theta-\theta_\star)'\frac{X'X}{n}(\theta-\theta_\star).\]
As a result, if $\theta-\theta_\star\in\Theta_\star$, $\L_{n,\theta}(z)\geq -(n/8)\bar\kappa_1(s_\star)\|\theta-\theta_\star\|_2^2$. Hence $\PP^{(n)}(Z\notin\hat\e_{n,1}(\Theta_\star,\bar L))=0$, for $\bar L= n\bar\kappa_1(s_\star)/4$. Finally, \;\; we \;\; have \;\; $\nabla \log q_{n,\theta_\star}(Z) = \sum_{i=1}^n \left(Z_i - g'(\pscal{x_i}{\theta_\star})\right)x_i$, and by Hoeffding's inequality, and a standard union bound argument,
\begin{eqnarray*}
\PP^{(n)}\left(Z\notin \e_{n,0}(\rset^d,\rho)\right)& = &\PP^{(n)}\left(\max_{1\leq j\leq d}\left|
\sum_{i=1}^n\left(Z_i - g'(\pscal{x_i}{\theta_\star})\right)X_{ij}\right|>\frac{\rho}{2}\right)\\
&\leq & 2\exp\left(\log(d) -\frac{\rho^2}{8\|X\|_\infty^2 n}\right) = \frac{2}{d},
\end{eqnarray*}
given the choice of $\rho$ in (\ref{rho:logit}) of the main paper. Hence we can apply Theorem \ref{thm2}-Part(1). This says that for any $k\geq 0$, 
\begin{multline*}
\PE^{(n)}\left[\check\Pi_{n,d}\left(\{\theta\in\rset^d:\;\|\theta\|_0\geq s_\star+ k\}\vert Z\right)\right] \leq \frac{2}{d} \\
+ 2e^{-\textsf{a}}\left(4+\frac{\bar\kappa_1(s_\star)}{16\|X\|_\infty^2\log(d)}\right)^{s_\star}{d \choose s_\star}\left(\frac{4c_2}{d^{c_4}}\right)^k,\end{multline*} 
where $\textsf{a}=(1/2)\inf_{x>0} \left[\r(x) -4\rho s_\star^{1/2} x\right]$. It is not hard to verify that for $\tau,b,c>0$, $\inf_{x>0}\left[\frac{\tau x^2}{1+bx}-cx\right]\geq -\frac{c^2}{4\sqrt{\tau}\sqrt{\tau-cb}}\geq -\frac{c^2}{2\tau}$, if $\tau\geq (4/3)bc$. In the case of $\textsf{a}$, the condition $\tau\geq (4/3)bc$ is satisfies if $\sqrt{n} \geq 64\times(4/3) \|X\|_\infty^2 s_\star\sqrt{\log(d)}/\underline{\kappa}_1$, and we have
\[\textsf{a}\geq -\frac{64\|X\|_\infty^2 s_\star\log(d)}{\underline{\kappa}_1}.\]
Using 
%B\ref{B2}, and for $k\geq s_\star$,
%\begin{multline*}
%\frac{1}{\pi_{\delta_\star}}\sum_{\delta:\;\|\delta\|_0>k}\pi_\delta 4^{\|\delta\|_0} = \frac{{d\choose s_\star}}{g_{s_\star}} \sum_{j=k+1}^d 4^j g_j 
%\leq  \frac{{d\choose s_\star}}{g_{s_\star}} \sum_{j=k+1}^d 4^j \left(\frac{c_2}{d^{c_4}}\right)^{j-s_\star}g_{s_\star} \\
%= {d\choose s_\star} 4^{s_\star} \sum_{j=k+1}^d \left(\frac{4c_2}{d^{c_4}}\right)^{j-s_\star}.\end{multline*}
%Therefore, if $d$ is large enough so that $d^{c_4}\geq 8c_2$, then for $k\geq s_\star$,
%\begin{multline*}
%\frac{1}{\pi_{\delta_\star}}\sum_{\delta:\;\|\delta\|_0>k}\pi_\delta 4^{\|\delta\|_0}\leq 2{d\choose s_\star}4^{s_\star} \left(\frac{4c_2}{d^{c_4}}\right)^{k-s_\star+1}\\
%\leq 2\exp\left(s_\star\log(4) +s_\star\log(de) +(k+1-s_\star)\log\left(\frac{4c_2}{d^{c_4}}\right)\right),\end{multline*}
this and the combinatorial inequality ${d \choose s}\leq e^{s\log(de)}$, it follows that 
\begin{multline*}
\PE^{(n)}\left[\check\Pi_{n,d}\left(\{\theta\in\rset^d:\;\|\theta\|_0\geq s_\star+k\}\vert Z\right)\right] \leq \frac{2}{d} \\
+ 2\exp\left[s_\star\log(d)\left(1+\frac{64\|X\|_\infty^2}{\underline{\kappa}_1} +\frac{\bar\kappa_1(s_\star)}{64\|X\|_\infty^2\log(d)^2} +\frac{\log(4e)}{\log(d)}\right) + k\log\left(\frac{4c_2}{d^{c_4}}\right)\right].\end{multline*}
Then for $\alpha>0$, choose 
\begin{equation}\label{eq:k:ising}
k= \frac{2\alpha}{c_4} +\frac{2}{c_4}\left(1 + \frac{64\|X\|_\infty^2}{\underline{\kappa}_1} +\frac{\bar\kappa_1(s_\star)}{64\|X\|_\infty^2(\log(d))^2}+\frac{\log(4e)}{\log(d)}\right)s_\star,\end{equation}
to conclude that the second term on the right-hand side of the above inequality is upper-bounded by $\frac{2}{d^\alpha}$, provided that $d^{c_4/2}\geq 4c_2$. Setting $\alpha=1$ proves the theorem.
\end{proof}

\begin{proof}[Proof of Part(2)]
We apply Theorem \ref{thm2}-Part(2) with $\bar\lambda=\rho\sqrt{\bar s}$ with $\rho$ as in (\ref{rho:logit}) of the main paper, and $\bar s = \zeta +s_\star$, with $\zeta$ as in Part (1). We also choose $\bar L =n\bar\kappa_1(s_\star)/4$, $\bar \Theta= \{\theta\in\rset^d:\;\|\theta-\theta_\star\|_0\leq \bar s\}$, the rate function $\r(x) =  n\underline{\kappa}_1(\bar s) x^2/(1+\sqrt{\bar s}\|X\|_\infty x/2)$, and $\e_n = \e_{n,0}(\bar\Theta,\bar\lambda)\cap \hat\e_{n,1}(\Theta_\star,\bar L)\cap \check\e_{n,1}(\bar\Theta,\r)$. With similar calculations as in Part (1), it is easy to establish that 
\[\PP^{(n)}(Z\notin\e_n)\leq \frac{2}{d}.\]  
If $\theta\notin\bar\Theta$, then $\|\theta\|_0>\bar s-s_\star=\zeta$, and by Part (1), we conclude that 
\[\PE^{(n)}\left[\check\Pi_{n,d}(\rset^d\setminus\bar\Theta\vert Z)\right]\leq \frac{4}{d}.\] 
Recall that $\phi_{\r}(a) = \inf\{x>0:\; \r(z)-a z\geq 0, \mbox{ for all } z\geq x\}$. Since $\r(x) = n\underline{\kappa}_1(\bar s) x^2/(1+\sqrt{\bar s}\|X\|_\infty x/2)$, if $n\underline{\kappa}_1(\bar s)-\bar s^{1/2}\bar\lambda\|X\|_\infty>0$, then 
\[\bar\epsilon =\phi_\r(2\bar\lambda) = \frac{\bar\lambda}{n\underline{\kappa}_1(\bar s)-\bar s^{1/2}\bar\lambda\|X\|_\infty}.\]
Then we take $n$ large enough so that $(3/4)n\underline{\kappa}_1(\bar s)\geq \bar s^{1/2}\bar\lambda \|X\|_\infty$, to conclude that 
\[\bar\epsilon = \frac{\bar\lambda}{n\underline{\kappa}_1(\bar s)-\bar s^{1/2}\bar\lambda\|X\|_\infty} \leq \frac{4\bar\lambda}{n\underline{\kappa}_1(\bar s)}= \frac{16\|X\|_\infty}{\underline{\kappa}_1(\bar s)} \sqrt{\frac{\bar s\log(d)}{n}}<\infty.\]
The condition $(3/4)n\underline{\kappa}_1(\bar s)\geq \bar s^{1/2}\bar\lambda \|X\|_\infty$ translates into the sample size condition $\sqrt{n}\geq (16/3)\|X\|_\infty^2(\bar s/\underline{\kappa}_1(\bar s))\sqrt{\log(d)}$, which holds by assumption. We fix $M_0\geq \max(500,1+(c_3+c_4/2)/8)$, and  apply Theorem \ref{thm2} to get:
\begin{multline}\label{eq1:proof:thm1:logit}
\PE^{(n)}\left[\check\Pi_{n,d}\left(\left\{\theta\in\rset^d:\;\|\theta-\theta_\star\| >M_0\bar\epsilon\right\}\vert Z\right)\right]\leq \frac{6}{d} + \sum_{j\geq 1}\textsf{D}_je^{-\frac{1}{8}\r(\frac{j M_0 \bar\epsilon}{2})} \\
+2{d\choose s_\star}\left(\frac{d^{c_3}}{c_1}\right)^{s_\star}\left(1+\frac{\bar L}{\rho^2}\right)^{s_\star}\sum_{j\geq 1}e^{3\rho \bar s^{1/2}jM_0\bar\epsilon}e^{-\frac{1}{8}\r(\frac{j M_0 \bar\epsilon}{2})}.\end{multline}
 Since  $\phi_\r(a)$ is defined as $\inf\{x>0:\; \r(z)\geq az,\;\mbox{ for all } z\geq x\}$, and $jM_0\bar\epsilon/2\geq \bar\epsilon=\phi_\r(2\bar\lambda)$, we have $\r(jM_0\bar\epsilon/2)\geq 2\bar\lambda(jM_0\bar\epsilon/2) = \rho\sqrt{\bar s}jM_0\bar\epsilon$. Hence
\begin{equation}\label{eq11:proof:thm1:logit}
\sum_{j\geq 1} e^{-\frac{1}{8}\r\left(\frac{jM_0\bar\epsilon}{2}\right)}\leq \sum_{j\geq 1} e^{-\frac{1}{8}jM_0\sqrt{\bar s}\rho\bar\epsilon}= \frac{e^{-\frac{1}{8}M_0\sqrt{\bar s}\rho\bar\epsilon}}{1-e^{-\frac{1}{8}M_0\sqrt{\bar s}\rho\bar\epsilon}}\leq 2e^{-8M_0\bar s\log(d)}, \end{equation}
where the last inequality follows from the bounds 
\[\frac{1}{8}M_0\sqrt{\bar s}\rho\bar\epsilon \geq \frac{1}{8}M_0\sqrt{\bar s}\rho \left(\frac{\bar\lambda}{n\underline{\kappa}_1(\bar s)}\right) =2M_0 \frac{\bar s \|X\|_\infty^2}{\underline{\kappa}_1(\bar s)}  \log(d)\geq 8M_0 \bar s\log(d)\geq 1\]
 since $8M_0\bar s \geq 16M_0/c_4\geq 1$, and $\log(d)\geq 1$, by assumption. Using the arguments in Example 7.1 of \cite{ghosal:etal:00} shows that the packing numbers $\textsf{D}_j$ satisfies $\sup_{j\geq 1}\textsf{D}_j\leq {d\choose \bar s}(24)^{\bar s}\leq (24)^{\bar s}e^{\bar s\log(d e)}$. It follows that 
\begin{eqnarray*}
\sum_{j\geq 1}\textsf{D}_j e^{-\frac{1}{8}\r(\frac{j M_0 \bar\epsilon}{2})}  &\leq & 2\exp\left[\bar s\log(d)\left(1+\frac{\log(24e)}{\log(d)}-8 M_0\right)\right] \\
&\leq & \frac{2}{d},
\end{eqnarray*}
provided that $\log(d)\geq 1$, and using the condition $8M_0 \geq c_4/2 + 1 + \log(24e)$. Setting  $x=jM_0\bar\epsilon/2$, we have
\begin{eqnarray}\label{eq12:proof:thm1:logit}
3\rho\sqrt{\bar s}jM_0\bar\epsilon-\frac{1}{8}\r(\frac{j M_0 \bar\epsilon}{2})&\leq& -\frac{x}{8}\left(\frac{n\underline{\kappa}_1(\bar s) x}{1+\frac{1}{2}\sqrt{\bar s}\|X\|_\infty x} - 48\rho\sqrt{\bar s}\right),\nonumber\\
&\leq & -\frac{x}{8}\left(\frac{n\underline{\kappa}_1(\bar s) \frac{M_0\bar\epsilon}{2}}{1+\frac{1}{2}\sqrt{\bar s}\|X\|_\infty \frac{M_0\bar\epsilon}{2}} - 48\rho\sqrt{\bar s}\right)\nonumber\\
&\leq &-\frac{2\rho\sqrt{\bar s}x}{8},
\end{eqnarray}
provided that 
\[\frac{n\underline{\kappa}_1(\bar s) \frac{M_0\bar\epsilon}{2}}{1+\frac{1}{2}\sqrt{\bar s}\|X\|_\infty \frac{M_0\bar\epsilon}{2}} - 48\rho\sqrt{\bar s} \geq 2\rho\sqrt{\bar s}.\]
This latter condition holds for all $M_0\geq 500$, if $\sqrt{n}> 125\bar s\|X\|_\infty^2\sqrt{\log(d)}/\underline{\kappa}_1(\bar s)$. In which case, from (\ref{eq12:proof:thm1:logit}) we have
\[\sum_{j\geq 1}e^{3\rho \bar s^{1/2}jM_0\bar\epsilon}e^{-\frac{1}{8}\r(\frac{j M_0 \bar\epsilon}{2})} \leq \sum_{j\geq 1} e^{-\frac{1}{8}jM_0\sqrt{\bar s}\rho\bar\epsilon} \stackrel{(a)}{\leq} 2e^{-8M_0\bar s\log(d)},\]
where the inequality (a) uses (\ref{eq11:proof:thm1:logit}).  In conclusion, the last term on the right-hand side of (\ref{eq1:proof:thm1:logit}) is upper-bounded by
\begin{multline}\label{eq2:proof:thm1:logit}
4{d\choose s_\star} \left(\frac{d^{c_3}}{c_1}\right)^{s_\star}\left(1+\frac{\bar L}{\rho^2}\right)^{s_\star}e^{-8M_0\bar s\log(d)} \leq4\exp\left[s_\star\log(d)\left( 1+c_3 +\frac{\log(e/c_1)}{\log(d)} \right.\right.\\
\left.\left. + \frac{\bar\kappa_1(s_\star)}{64\|X\|_\infty^2\log(d)^2}\right)-8M_0\bar s\log(d)\right].
\end{multline}
Given  that $\bar s =s_\star +\zeta$ with $\zeta$ as in Part (1), since $\log(d)\geq \log(e/c_1)$, and $8M_0 \geq 2+c_3$, we see that the  right-side of (\ref{eq2:proof:thm1:logit}) is upper-bounded by $4(1/d)^{16M_0/c_4}\leq (4/d)$. The theorem follows.
\end{proof}

\subsection{Proof of Theorem \ref{thm:Ising:1} and \ref{thm:Ising:3}}\label{sec:proof:thm:Ising}
It is obvious that H\ref{H0} and H\ref{H1} hold for this example. For convenience in the notation, for $z\in\rset^{n\times p}$, $1\leq j\leq p$, we let $z^{(j)}\in\rset^{n\times p}$ be the matrix obtained by replacing all the components of the $j$-th column of $z$ by $1$. We introduce
\begin{multline*}
q^{(j)}_{n,u}(z) \eqdef \prod_{i=1}^n \frac{\exp\left(z_{ij}\left(u_j +\sum_{k\neq j} u_k z_{ik}\right)\right)}{1+\exp\left(u_j + \sum_{k\neq j} u_k z_{ik}\right)} ,\\
\;\;\mbox{ and }\;\; \L_{n,u}^{(j)}(z) \eqdef \log q^{(j)}_{n,u}(z) - \log q^{(j)}_{n,\theta_{\star \cdot j}}(z) -\pscal{\nabla\log q^{(j)}_{n,\theta_{\star \cdot j}}(z)}{u-\theta_{\star \cdot j}},\;\;\;u\in\rset^p.\end{multline*}
The function $u\mapsto q^{(j)}_{n,u}(z)$ is the likelihood function of the logistic regression model of the $j$-column of $z$ on $z^{(j)}$. Let $\H_n^{(j)}(z)\eqdef-\nabla^{(2)}\log q_{n,\theta_{\star \cdot j}}(z)$. Specifically, we have
\[(\H_n^{(j)}(z))_{st} \eqdef \sum_{i=1}^n g^{(2)}\left(\theta_{\star jj} +\sum_{k\neq j} z_{ik}\theta_{\star kj}\right)z^{(j)}_{i s}z^{(j)}_{i t},\;\;1\leq s,t\leq p.\]
We will need the following restricted smallest eigenvalues of $\H_n^{(j)}(z)$.
\[\underline{\kappa}_2^{(j)}(z) \eqdef \inf\left\{\frac{u'(\H_n^{(j)}(z))u}{n\|u\|^2},\; u\in\rset^p\setminus\{0\},\; \sum_{k:\;\delta_{\star,kj}=0}|u_k|\leq 7\sum_{k:\;\delta_{\star,kj}=1}|u_k|.\right\}.\]
\[\underline{\kappa}_2^{(j)}(s,z) \eqdef \inf\left\{\frac{u'(\H_n^{(j)}(z))u}{n\|u\|^2},\;u\in\rset^p\setminus\{0\},\; \|u\|_0\leq s.\right\}.\]
\[\underline{\kappa}_2(z)=\inf_{1\leq j\leq p} \underline{\kappa}_2^{(j)}(z),\;\;\mbox{ and }\;\; \underline{\kappa}_2(s,z)=\inf_{1\leq j\leq p} \underline{\kappa}_2^{(j)}(s,z).\]

The next result shows that if $\underline{\kappa}_2(s)>0$ and $\underline{\kappa}_2>0$ (with $\underline{\kappa}_2(s)$ and $\underline{\kappa}_2$ as defined in (\ref{def:kappa:ising}) of the main paper, then with high probability $\underline{\kappa}_2(Z)>0$ and $\underline{\kappa}_2(s,Z)>0$. The proof is an easy modification of the argument of \cite{atchade:ejs14}~Lemma 2.5. We omit the details.

\begin{lemma}\label{lem:tech_lem:Ising}
Assume C\ref{C1}. There exist finite universal constants $a_1, a_2$ such that the following two statements holds true.
\begin{enumerate}
\item For $1\leq s\leq p$, if $\underline{\kappa}_2(s)>0$, and $n\geq a_1\left(\frac{s}{\underline{\kappa}_2(s)}\right)^2\log(p)$, then
\[\PP^{(n)}\left(\underline{\kappa}_2(s,Z)\leq \frac{\underline{\kappa}_2(s)}{2}\right)\leq e^{-a_2 n}.\]
\item If $\underline{\kappa}_2>0$, and $n\geq a_1\left(\frac{s_\star}{\underline{\kappa}}\right)^2\log(p)$, then
\[\PP^{(n)}\left(\underline{\kappa}_2(Z)\leq \frac{\underline{\kappa}_2}{2}\right)\leq e^{-a_2 n}.\]
\end{enumerate}
\end{lemma}

\begin{proof}[Proof of Theorem \ref{thm:Ising:1}-Part(1)]
We will reduce the result to Theorem \ref{thm1:logit}~Part(1). We set 
\begin{equation*}\label{def:setG}
\G^{(j)}\eqdef \{z\in\rset^{n\times p}:\; \underline{\kappa}_2^{(j)}(z)> \underline{\kappa}_2/2\},\;\;\mbox{ and }\;\; \G\eqdef \{z\in\rset^{n\times p}:\; \underline{\kappa}_2(z)> \underline{\kappa}_2/2\}.\end{equation*}
We also define $\A^{(j)}\eqdef \{u\in\rset^p:\;\|u\|_0\leq \zeta_j\}$. Define $\Theta\eqdef \{\theta\in\rset^{p\times p}:\; \|\theta_{\cdot j}\|_0\leq \bar s_j,\;1\leq j\leq p\}$. Hence if $\theta\notin\Theta$, then $\theta_{\cdot j}\notin \A^{(j)}$, for some $j$. Therefore,
\[\PE^{(n)}\left[\check\Pi_{n,d}\left(\rset^{d\times d}\setminus \Theta\vert Z\right)\right] \leq \PP^{(n)}\left(Z\notin \G\right) + \sum_{j=1}^p \PE^{(n)}\left[\textbf{1}_{\G}(Z)\check\Pi_{n,d,j}\left(\rset^d\setminus \A^{(j)}\vert Z\right)\right].\]
Note that $\G\subseteq\G^{(j)}$, and $\{Z\in \G^{(j)}\}$ is $Z^{(j)}-$measurable. Hence by conditioning on $Z^{(j)}$, we get
\begin{multline*}
\PE^{(n)}\left[\check\Pi_{n,d}\left(\rset^{d\times d}\setminus \Theta\vert Z\right)\right] \leq \PP^{(n)}\left(Z\notin \G\right) \\
+ \sum_{j=1}^p \PE^{(n)}\left[\textbf{1}_{\G^{(j)}}(Z)\PE^{(n)}\left[\check\Pi_{n,d,j}\left(\rset^d\setminus \A^{(j)}\vert Z\right)\vert Z^{(j)}\right]\right]
\end{multline*}
By conditioning on $Z^{(j)}$, and  for $Z\in\G^{(j)}$, we are taken back to the setting of the standard logistic regression with a well-behaved design matrix. With the choice of $\zeta_j$, and since $\rho$ in (\ref{rho:ising}) of the main paper is taken larger than $4\sqrt{n\log(p)}$, by Theorem \ref{thm1:logit} (1), there exists an absolute constant $A_1$ such that for $p^{c_4}\geq 8c_2\max(1,2c_2)$, and  $n\geq A_1(s_\star/\underline{\kappa}_2)^2\log(p)$, we have
\[\PE^{(n)}\left[\check\Pi_{n,d,j}\left(\rset^d\setminus \A^{(j)}\vert Z\right)\vert Z^{(j)}\right]\ \leq \frac{4}{p^2}.\]
The term $p^2$ in $4/p^2$ comes from using $\alpha=2$ in (\ref{eq:k:ising}). Without any loss of generality  we can take $A_1$ as large as the constant $a_1$ in Lemma \ref{lem:tech_lem:Ising} to conclude that  $\PP^{(n)}\left(Z\notin \G\right)\leq e^{-a_2 n}$. Hence
\[\PE^{(n)}\left[\check\Pi_{n,d}\left(\rset^{d\times d}\setminus \Theta\vert Z\right)\right] \leq e^{-a_2 n} + \frac{4}{p},\]
as claimed.
\end{proof}

\begin{proof}[Proof of Theorem \ref{thm:Ising:1}-Part(2)]
We shall apply Theorem \ref{thm2}-Part(2). We will apply the theorem with the split cone 
\[\bar\Theta \eqdef \{\theta\in\rset^{d\times d}:\; \|\theta_{\cdot j}\|_0 \leq \bar s_j,\; 1\leq j\leq p\}.\]
Here the norm $\|\cdot\|_2$ in Theorem \ref{thm2} is the Frobenius norm $\normfro{\cdot}$, whereas the notation $\|\cdot\|_2$ in what follows will denote the Euclidean norm on $\rset^p$. Notice that if $\theta\notin\theta_\star +\bar\Theta$, then $\theta\notin \Theta$ (where $\Theta$ is as defined in Theorem \ref{thm:Ising:1}). Hence we will use Theorem \ref{thm:Ising:1} to control the term $\PE^{(n)}\left(\check\Pi_{n,p}(\rset^{d\times d}\setminus (\theta_\star +\bar\Theta)\vert Z)\right)$. More precisely, there exist universal positive constants $A_1, A_2$ such that for $p^{c_4}\geq 8c_2\max(1,2c_2)$, and $n\geq A_1(s_\star/\underline{\kappa}_2)^2\log(p)$,
\begin{equation}\label{eq0:proof:ising}
\PE^{(n)}\left(\check\Pi_{n,p}(\rset^{d\times d}\setminus (\theta_\star +\bar\Theta)\vert Z)\right)
\leq e^{-A_2n} +\frac{4}{p}.\end{equation}

Set $\bar S\eqdef \sum_{j=1}^p\bar s_j$, $\bar\lambda = \rho\sqrt{\bar S}$, $\bar L=ns_\star/4$, $\r(x) =n\underline{\kappa}(\bar s)x^2/(2+\bar S^{1/2}x)$, and consider $\e_n=\e_{n,0}(\bar\Theta,\bar\lambda)\cap \hat\e_{n,1}(\Theta_\star,\bar L)\cap\check\e_{n,1}(\bar\Theta,\r)$.
We have
\[\sup_{u\in\bar\Theta,\;\normfro{u}=1}\left|\pscal{\nabla \log q_{n,\theta_\star}(Z)}{u}_{\textsf{F}}\right|\leq \sqrt{\sum_{j=1}^p\bar s_j}\|\nabla\log q_{n,\theta_\star}(Z)\|_\infty.\]
Using this and a standard Hoeffding inequality, we obtain that 
\begin{equation}\label{eq:control:e_0}
\PP^{(n)}\left(Z\notin \e_{n,0}(\bar\Theta,\bar\lambda)\right)\leq 2\exp\left(2\log(p)-\frac{1}{2n}\left(\frac{\bar\lambda}{2\sqrt{\sum_{j=1}^p\bar s_j}}\right)^2\right) \leq \frac{2}{p},\end{equation}
given the choice of $\bar\lambda$, and $\rho$ in (\ref{rho:ising}).

We use a second order Taylor expansion of $u\mapsto q^{(j)}_{n,u}(z)$ around $\theta_{\star \cdot j}$ and the fact that $g^{(2)}(x)\leq 1/4$ to deduce that for all $\theta\in\theta_\star + \Theta_\star$
\[\L_{n,\theta}(z) =-\frac{n}{8}\sum_{j=1}^p(\theta_{\cdot j}-\theta_{\star \cdot j})'\left(\frac{[Z^{(j)}]'[Z^{(j)}]}{n}\right)(\theta_{\cdot j}-\theta_{\star \cdot j})\geq -\frac{n s_\star}{8}\normfro{\theta-\theta_\star}^2.\]
Hence with $\bar L=ns_\star/4$, 
\begin{equation}\label{eq:control:he_1}
\PP^{(n)}(Z\notin \hat\e_{n,1}(\Theta_\star,\bar L))=0.\end{equation}

Consider the set 
\[\G^{(j)}\eqdef \{z\in\rset^{n\times p}:\; \underline{\kappa}^{(j)}(\bar s, z)> \underline{\kappa}(\bar s)/2\},\;\;\mbox{ and }\;\; \G\eqdef \{z\in\rset^{n\times p}:\; \underline{\kappa}(\bar s,z)> \underline{\kappa}(\bar s)/2\}.\]
Take $Z\in\G$. Then for all $j$, $\underline{\kappa}(\bar s, Z^{(j)})>\underline{\kappa}(\bar s)/2$ and we can use the same argument in (\ref{upper:bound:L}) to conclude that for $\theta-\theta_\star\in\bar\Theta$,
\[\L^{(j)}_{n,\theta_{\cdot j}}(Z) \leq -\frac{n\underline{\kappa}(\bar s)\|\theta_{\cdot j}-\theta_{\star \cdot j}\|_2^2/2}{2+\sqrt{\bar s_j}\|\theta_{\cdot j}-\theta_{\star \cdot j}\|_2}.\]
It follows that for $\theta-\theta_\star\in\bar\Theta$,
\[\L_{n,\theta}(Z)\leq -\sum_{j=1}^p \frac{n\underline{\kappa}(\bar s)\|\theta_{\cdot j}-\theta_{\star \cdot j}\|_2^2/2}{2+\sqrt{\bar s_j}\|\theta_{\cdot j}-\theta_{\star \cdot j}\|_2}\leq -\frac{1}{2}\frac{n\underline{\kappa}(\bar s)\normfro{\theta-\theta_\star}^2}{2+\bar S^{1/2}\normfro{\theta-\theta_\star}} = -\frac{1}{2}\r(\normfro{\theta-\theta_\star}).\]
Hence, with the rate function $\r(x) =n\underline{\kappa}(\bar s)x^2/(2+\bar S^{1/2}x)$, we have
\begin{equation}\label{eq:control:ce_1}
\PP^{(n)}\left(Z\notin \check\e_{n,1}(\bar\Theta,\r)\right)\leq \PP^{(n)}\left(Z\notin \G\right) \leq e^{-a_2n},\end{equation}
as seen in Lemma \ref{lem:tech_lem:Ising}, provided that $n\geq A_1\left(\frac{\bar s}{\underline{\kappa}(\bar s)}\right)^2\log(p)$ (without any loss of generality, we take $A_1$ greater than the constant $a_1$ in Lemma \ref{lem:tech_lem:Ising}). Hence, with $\e_n= \e_{n,0}(\bar\Theta,\bar\lambda)\cap \hat\e_{n,1}(\Theta_\star,\bar L)\cap \check\e_{n,1}(\bar\Theta,\r)$, it follows from (\ref{eq:control:e_0})-(\ref{eq:control:ce_1}) that for $n\geq A_1\left(\frac{\bar s}{\underline{\kappa}(\bar s)}\right)^2\log(p)$
\begin{equation}\label{control:e:ising}
\PP^{(n)}\left(Z\notin \e_n\right)\leq e^{-a_2n} + \frac{2}{p}.
\end{equation}

Finally, we note that with the same calculations as in the proof of Theorem \ref{thm1:logit}~Part(2), we can choose the constant $a_1$ such that for $n\geq A_1\left(\frac{\bar S}{\underline{\kappa}(\bar s)}\right)^2\log(p)$,
\[\frac{2\bar\lambda}{n\underline{\kappa}(\bar s)}\leq \bar\epsilon =\phi_\r(2\bar\lambda)\leq \frac{4\bar\lambda}{n\underline{\kappa}(\bar s)} \leq \frac{96}{\underline{\kappa}(\bar s)}\sqrt{\frac{\bar S\log(p)}{n}}<\infty.\]
We are then ready to apply Theorem \ref{thm2}-Part(2). Fix $M_0\geq \max(500,1 + (c_3+c_4/2)/8)$, set $V\eqdef \{\theta\in\rset^{d\times d}:\;\normfro{\theta-\theta_\star}>M_0\bar\epsilon\}$, then for $n\geq A_1\left(\frac{\bar S}{\underline{\kappa}(\bar s)}\right)^2\log(p)$, (\ref{thm2:maineq}), (\ref{eq0:proof:ising}), and (\ref{control:e:ising}) give
\begin{multline}\label{eq:proof:thm:ising}
\PE^{(n)}\left[\check\Pi_{n,d}(V\vert Z)\right]\leq \left(e^{-A_2 n} +\frac{4}{p}\right) + \left(e^{-a_2n}+\frac{2}{p}\right) + \sum_{j\geq 1}\textsf{D}_je^{-\frac{1}{8}\r(\frac{j M_0 \bar\epsilon}{2})}\\
+\left\{\prod_{j=1}^p2{p\choose s_{\star j}} \left(\frac{p^{c_3}}{c_1}\right)^{s_{\star j}}\left(1+\frac{\bar L}{\rho^2}\right)^{s_{\star j}}\right\}  \sum_{k\geq 1}e^{-\frac{1}{8}\r(\frac{k M_0 \bar\epsilon}{2})} e^{3\rho c_0 k M_0\bar\epsilon}.\end{multline}
Similar calculations as in the proof of Theorem \ref{thm1:logit}~Part(2) shows that 
\[\sum_{j\geq 1}\textsf{D}_je^{-\frac{1}{8}\r(\frac{j M_0 \bar\epsilon}{2})} \leq \frac{2}{p},\;\mbox{ and }\;\; \sum_{j\geq 1}e^{-\frac{1}{8}\r(\frac{j M_0 \bar\epsilon}{2})} e^{3\rho c_0 j M_0\bar\epsilon}\leq 2 e^{-16M_0\bar S\log(p)},\]
and
\begin{multline*}
\prod_{j=1}^p2{p\choose s_{\star j}} \left(\frac{p^{c_3}}{c_1}\right)^{s_{\star j}}\left(1+\frac{\bar L}{\rho^2}\right)^{s_{\star j}}\\
\leq 2^p\exp\left(\left(\sum_{j=1}^p s_{\star j}\right)\log(p) \left(1+c_3 +\frac{\log(e/c_1)}{\log(p)} +\frac{s_\star}{4(24^2)\log(p)^2}\right)\right).\end{multline*}
Hence, and by the same argument as in the proof of Theorem \ref{thm1:logit}~Part(2), the last term on the right-side of (\ref{eq:proof:thm:ising}) is bounded by $4/p$.
\end{proof}

\begin{proof}[Proof of Theorem \ref{thm:Ising:3}]
We will reduce this result to Theorem \ref{thm1:logit}~Part(2). We set 
\[\mathcal{V}\eqdef\{\theta\in\rset^{p\times p}:\; \|\theta_{\cdot j}\|_2> \epsilon_j,\;\mbox{ for some } j\},\]
and $\bar{\mathcal{V}} \eqdef \bar\Theta \cap \mathcal{V}$, where $\bar\Theta = \{\theta\in\rset^{d\times d}:\; \|\theta_{\cdot j}\|_0 \leq \bar s_j,\; 1\leq j\leq p\}$. Using Theorem \ref{thm:Ising:1} as we did in (\ref{eq0:proof:ising}), there exist universal positive constants $A_1, A_2$ such that for $p^{c-4}\geq 8c_2\max(1,2c_2)$, and $n\geq A_1(s_\star/\underline{\kappa}_2)^2\log(p)$,
\begin{eqnarray*}
\PE^{(n)}\left[\check\Pi_{n,d}(\mathcal{V}\vert Z)\right] &\leq & \PE^{(n)}\left[\check\Pi_{n,d}(\rset^{d\times d}\setminus\bar\Theta\vert Z)\right] + \PE^{(n)}\left[\check\Pi_{n,d}(\bar{\mathcal{V}}\vert Z)\right],\\
& \leq & e^{-A_2 n} + \frac{4}{p} + \PE^{(n)}\left[\check\Pi_{n,d}(\bar{\mathcal{V}}\vert Z)\right].\end{eqnarray*}
We define
\begin{equation*}
\G^{(j)}\eqdef \{z\in\rset^{n\times p}:\; \underline{\kappa}_2^{(j)}(z)> \underline{\kappa}_2/2\},\;\;\mbox{ and }\;\; \G\eqdef \{z\in\rset^{n\times p}:\; \underline{\kappa}_2(z)> \underline{\kappa}_2/2\}.\end{equation*}
We also define $\A^{(j)}\eqdef \{u\in\rset^p:\;\|u\|_0\leq \bar s_j,\;\mbox{ and }\; \|u\|_2>\epsilon_j\}$. Hence, if $\theta\in\bar{\mathcal{V}}$, then $\theta_{\cdot j}\in \A^{(j)}$, for some $j$. Therefore,
\begin{multline*}
\PE^{(n)}\left[\check\Pi_{n,d}(\bar{\mathcal{V}}\vert Z)\right] \leq \PP^{(n)}\left(Z\notin \G\right) + \sum_{j=1}^p \PE^{(n)}\left[\textbf{1}_{\G^{(j)}}(Z)\PE^{(n)}\left[\check\Pi_{n,d,j}\left(\A^{(j)}\vert Z\right)\vert Z^{(j)}\right]\right]
\end{multline*}
Fix $M_0\geq \max(125,c_4(1+c_3)/64)$. $\PE^{(n)}\left[\check\Pi_{n,d,j}\left(\A^{(j)}\vert Z\right)\right]$ is the same as the posterior distribution of the logistic regression of the $j$-th column of $Z$ on $Z^{(j)}$, and for $Z^{(j)}\in \G^{(j)}$, we can apply Theorem \ref{thm1:logit}~Part(2). Hence, we can take $A_1$ large enough so that for $p\geq e(1+c_1)/c_1$, and $n\geq A_1(\bar s/\underline{\kappa}_2)^2\log(p)$,
\[\textbf{1}_{\G^{(j)}}(Z)\PE^{(n)}\left[\check\Pi_{n,d,j}\left(\A^{(j)}\vert Z\right)\vert Z^{(j)}\right]\leq \frac{8}{p^2}.\]
Hence
\[\PE^{(n)}\left[\check\Pi_{n,d}(\mathcal{V}\vert Z)\right] \leq 2e^{-A_2 n} + \frac{12}{p}.\]
\end{proof}

\vspace{2cm}

\section*{Acknowledgements}
The author would like to thank Shuheng Zhou for very helpful conversations on restricted convexity.

\vspace{2cm}

%\begin{supplement}
%\sname{Supplement A}\label{suppA}
%\stitle{Supplement to ``On the contraction properties of some high-dimensional quasi-posterior distributions"}
%\slink[url]{http://www.e-publications.org/ims/support/dowload/imsart-ims.zip}
%\sdescription{The supplementary material contains the proof of Theorem \ref{thm1:logit}, \ref{thm:Ising:1} and \ref{thm:Ising:3}.}
%\end{supplement}

\end{document}